\documentclass[11pt,oneside]{amsart}
\usepackage{amssymb}

\usepackage{amsfonts,amsmath,amstext,amsbsy,euscript,amsthm}
\usepackage[margin=1in]{geometry}
\usepackage{mathabx}
\usepackage[all]{xy}
\usepackage[utf8]{inputenc} 
\usepackage{marginnote}
\usepackage{color}

\theoremstyle{plain}
\newtheorem{theorem}{Theorem}[section]
\newtheorem{proposition}[theorem]{Proposition}
\newtheorem{corollary}[theorem]{Corollary}
\newtheorem{lemma}[theorem]{Lemma}

\theoremstyle{definition}
\newtheorem{definition}[theorem]{Definition}

\theoremstyle{remark}
\newtheorem{remark}[theorem]{Remark}
\newcommand{\note}[2][\null]{%
  \marginpar{\renewcommand{\baselinestretch}{1}\vspace{-1em}\hrule\vspace{3pt}%
  \scriptsize\raggedright\textsf{#2\ifx#1\null\else\\\hfill--- 
  {\em #1}\fi}\vspace{1.5em}}%
}

\numberwithin{equation}{section}

\begin{document}

\title[Mean Value Operators On Symmetric Spaces]{Surjectivity of Mean Value Operators on Noncompact Symmetric Spaces}

\author{Jens Christensen}
\address{Department of Mathematics,
Colgate University,
Hamilton, NY 13346}
\email{jchristensen@colgate.edu}
\author{Fulton Gonzalez}
\address{Department of Mathematics,
Tufts University,
Medford, MA 02155}
\email{fulton.gonzalez@tufts.edu}
\author{Tomoyuki Kakehi}
\address{Department of Mathematics,
Faculty of Science,
Okayama  University,
700-8530 1-1-1 Tsushima-naka , Kita-ku , Okayama-shi, Japan}
\email{kakehi@math.okayama-u.ac.jp}


\date{\today}

\begin{abstract}
Let $X=G/K$ be a symmetric space of the non-compact type.
We prove
that 
the mean value operator
over translated $K$-orbits of a fixed point is surjective on the space of smooth functions on $X$
if $X$ is either complex or of rank one. For higher rank spaces it is shown that the same statement is
true for points in an appropriate Weyl subchamber.
\end{abstract}

\maketitle

\def\rar{\rightarrow} 
\def\fk{\mathfrak}
\def\e{\mathfrak e}
\def\g{\mathfrak g} 
\def\k{\mathfrak k} 
\def\a{\mathfrak a} 
\def\m{\mathfrak m} 
\def\n{\mathfrak n}
\def\p{\mathfrak p} 
\def\u{\mathfrak u} 
\def\t{\mathfrak t} 
\def\h{\mathfrak h} 
\def\z{\mathfrak z}
\def\d{\mathfrak d}
\def\s{\mathfrak s}
\def\Exp{\text{Exp}}
\def\dbar{\overline\delta}
\def\Ad{\text{Ad}}
\def\ad{\text{ad}}
\def\rr{\mathbb R}
\def\rn{\mathbb R^n}
\def\cn{\mathbb C^n}
\def\nm{\nonumber}
\def\cc{\mathbb C}
\def\zz{\mathbb Z}
\def\be{\begin{equation}}
\def\ee{\end{equation}}

\def\Bbb{\mathbb}
\def\Cal{\mathcal}

\section{Introduction}
Let $X=G/K$ be a symmetric space of the non-compact type.
We will investigate the question of surjectivity of convolution
operators on the space of smooth functions on $X$.
Let $\mathcal E(X)$ denote the space of smooth functions
on $X$ equipped with the topology of uniform convergence
of all derivatives on compact sets.
We will show, in particular,
that for any $y$ in $X$ the mean value operator
\begin{equation*}
M^yf(x)=\int_K f(gk\cdot y)\,dk\qquad(x=g\cdot o\in X),
\end{equation*}
is a surjective linear operator on $\mathcal E(X)$
if $X$ is either complex or of rank one.
The mean value operator above can be realized as a convolution
operator with a $K$-invariant distribution of compact support.
This allows us to transfer the problem to Euclidean
space via the Abel transform, and in Euclidean space
 we can apply conditions on the convolution kernel which 
were previously obtained by Ehrenpreis (\cite{Ehrenpreis1960}) and H{\"o}rmander (\cite{Ho2}). Furthermore, our approach allows us to conclude that
$G$-invariant differential operators are surjective on smooth functions on symmetric spaces, which is  one of the main results in \cite{He1}. 
(See also \cite{Ehrenpreis1954} or \cite{Malgrange1954}
for the corresponding result on Euclidean space.)

Mean value operators can also be thought of as Radon transforms related to double fibrations. (See \cite{GASS}, Ch. II, \S3). 
Some of the principal problems are to determine the kernels, ranges, and the mapping properties of these
transforms and their duals on spaces of  functions and distributions. For example, the dual classical Radon transform was 
shown to be surjective on the space of smooth functions 
by Hertle in 1984 (\cite{Hertle1984}).  A similar result was proved by Helgason for the dual horocycle Radon transform on symmetric 
spaces. (See \cite{Helgason1982} or \cite{GASS}, Ch. IV, Corollary 2.5.) 

Mean value operators are essentially self-dual integral transforms, so questions pertaining to these transforms and their duals coincide.
The surjectivity of spherical mean value operators on
Euclidean space as well as the hyperbolic space $\mathbb H^3$
was recently proved in the thesis of K. Lim \cite{Lim}.
The idea of using Ehrenpreis and H{\"o}rmander estimates
in our work originates from the thesis by K. Lim.

\section{The Ehrenpreis and H{\"o}rmander  Criteria}
Throughout this paper we will use the following standard spaces
on a smooth manifold $\mathcal M$. 
The space
$\mathcal E(\mathcal M)$ denotes the space of smooth
functions on $\mathcal M$ equipped with the topology of uniform
convergence of all derivatives on every compact subset of $\mathcal M$. 
The space 
$\mathcal D(\mathcal M)$ denotes the subspace of functions in
$\mathcal E(\mathcal M)$ which are compactly supported. 
The dual $\mathcal D'(\mathcal M)$ is called the space of distributions
on $\mathcal M$, and the dual $\mathcal E'(\mathcal M)$ is
the space of compactly supported distributions. 
The dual spaces $\mathcal D'(\mathcal M)$ and $\mathcal E'(\mathcal M)$
are (if not stated otherwise) equipped with the weak* topology.

In this section, we consider the convolution operator on $\mathbb R^n$ with a given distribution of compact support.  In particular, for a distribution $\mu\in\mathcal E'(\mathbb R^n)$, consider the convolution operator $c_\mu\colon\mathcal E(\mathbb R^n)\to\mathcal E(\mathbb R^n)$ given by
\begin{equation}\label{E:convolution1}
c_\mu(f)=f*\mu.
\end{equation}

Since $\mu$ has compact support, the Fourier-Laplace transform $\mu^*$ is a holomorphic function on $\mathbb C^n$.  A complete description of the Fourier transforms of compactly supported distributions is of course provided by the Paley-Wiener Theorem.  (See, for instance, Theorem 7.3.1 in \cite{Ho1}, where it is formulated in terms of support functions.)

The  operator $c_\mu$ is, in general,  not injective on $\mathcal E(\rn)$.  In fact suppose that $\zeta\in \mathbb C^n$ is a zero of the holomorphic function $\mu^*$. Then if $f(x)=e^{i\langle\zeta,x\rangle}$, we would have $c_\mu(f)=0$. However, if $\mu\neq 0$, then $c_\mu$ is injective as an operator on $\mathcal E'(\rn)$, as is easily seen by taking Fourier transforms.

On the other hand, $c_\mu\colon\mathcal E(\rn)\to\mathcal E(\rn)$ is often surjective, as for instance when $c_\mu$ happens to be a constant coefficient differential operator 
\cite[Theorem 10]{Ehrenpreis1954}.   In the theorem below, we collect conditions provided in H{ö}rmander's text  (\cite{Ho2}, Theorem 16.3.10, Definition 16.3.12, and Theorem 16.5.7) which are equivalent to the surjectivity of $c_\mu$.  
In order to formulate the theorem we need a definition:

\begin{definition}
We will say that a function $u:\cc^n\to \cc$ is \emph{slowly decreasing}
if there is a constant $A>0$ such that 
$$
\sup\{|u(\zeta)|\,\colon\,\zeta\in\mathbb C^n,\;\|\zeta-\xi\|\leq A\log(2+\|\xi\|)\}\geq (A+\|\xi\|)^{-A},
$$
for all $\xi\in\mathbb R^n$.
\end{definition}

\noindent This is easily shown to be equivalent to the following
seemingly more flexible condition:
the function $u:\cn\to\cc$ is slowly decreasing
if and only if there are positive constants $A,\,B,\,C$, and $D$ such that
\begin{equation}\label{E:invertibility2}
\sup\{|u(\zeta)|\,\colon\,\zeta\in\mathbb C^n,\;\|\zeta-\xi\|\leq A\log(2+\|\xi\|)\}\geq B(C+\|\xi\|)^{-D}
\end{equation}
\emph{for all $\xi\in\mathbb R^n$}.
\medskip

In what follows, let $\mu^\vee$ be the distribution $\mu^\vee(f)=\mu(f^\vee)$, for $f\in\mathcal E(X)$, where $f^\vee(x)=f(-x)$.
\begin{theorem}\label{T:invertible-dist}(Ehrenpreis \cite{Ehrenpreis1960}, H{ö}rmander \cite{Ho2})  Let $\mu\in\mathcal E'(\rn)$.  Then the following conditions on $\mu$ are equivalent.
\renewcommand{\theenumi}{(\roman{enumi})}
\renewcommand{\labelenumi}{\theenumi}
\begin{enumerate}
\item The convolution operator $c_\mu\colon\mathcal E(\rn)\to\mathcal E(\rn)$ is surjective.
\item\label{slow-decrease} The Fourier transform $\mu^*$ is slowly decreasing.
\item\label{invert-equiv} For any $T\in\mathcal E'(\mathbb R^n)$ such that $T^*/\mu^*$ is an entire function on $\mathbb C^n$, there is an $S\in\mathcal E'(\mathbb R^n)$ such that $S^*=T^*/\mu^*$.
\item The convolution operator $c_{\mu^\vee}\colon\mathcal E'(\rn)\to\mathcal E'(\rn)$ has weak$^*$ closed range.
\end{enumerate}
\end{theorem}

In the aforementioned reference, there are additional conditions on $\mu$ in case the domain and range of $c_\mu$ are required to have support in certain subsets of $\rn$, but we will not need them here since they are trivially satisfied when these subsets equal $\rn$.

Following H{ö}rmander, we will say that a distribution $\mu\in\mathcal E'(\mathbb R^n)$ is \emph{invertible}
provided that it  satisfies any of the equivalent conditions in Theorem \ref{T:invertible-dist}.

Since $c_{\mu^\vee}:\mathcal E'(\mathbb{R}^n)\to\mathcal E'(\mathbb{R}^n)$ is injective if $\mu\neq 0$, the equivalence of Conditions (i) and (iv) above is a special case of the following general fact about continuous linear mappings on Frech\'et spaces.  (In \cite{Ho2}, Theorem 16.5.7, it is used to prove that (iv) implies (i).)
\begin{theorem}\label{T:frechet}
Let $E$ and $F$ be Frech\'et spaces.  A continuous linear map $\Phi\colon E\to F$ is surjective if and only if its adjoint $\Phi^*\colon F'\to E'$ is injective and has a weak$^*$ closed range in $E'$.
\end{theorem}

\noindent For a proof of Theorem \ref{T:frechet}, see Theorem 7.7, Ch. IV in Schaefer's book \cite{Sch}.  See also Theorem 3.7, Ch. I of \cite{GASS} for a generalization.  We note that it is a straightforward consequence of the Hahn-Banach Theorem  that a subspace of $E'$ is  closed in the weak$^*$ topology of $E'$  if and only if it is closed in the strong topology of $E'$.


Note that while the other conditions in Theorem \ref{T:invertible-dist} are mapping conditions, Condition \ref{slow-decrease} is a condition that is in theory testable by computation.  According to \cite{Mal}, if the ratio $T^*/\mu^*$  is an entire function on $\mathbb C^n$, then it is of exponential type.  In order for this ratio to be the Fourier transform of a compactly supported distribution, it must  be of slow (i.e., polynomial) 
growth in $\mathbb R^n$, and the slow decrease condition \ref{slow-decrease} on $\mu^*$ is equivalent to this.

Finally, we note that if $\mu$ is invertible, then so is $\mu^\vee$, and the (closed) range $c_{\mu}(\mathcal E'(\rn))$ is given precisely by the set
\begin{equation}\label{E:conv-range}
\{S\in\mathcal E'(\rn)\,:\,S^*(\zeta)/\mu^*(\zeta)\text{ is an entire function on }\mathbb C^n\}.
\end{equation}

We will need the following refinement of Condition \ref{invert-equiv}, 
which can be found in the proof of Theorem 16.3.10 in \cite{Ho2}.


 \begin{proposition}\label{T:uniform-exp-type}  Suppose that $u:\cn\to\cc$ is slowly decreasing and holomorphic.
 For every triple $(C,R,N)\in\rr^+\times\rr^+\times\zz^+$, there is a triple $(C',R',N')\in\rr^+\times\rr^+\times\zz^+$ such that whenever $v:\cn\to\cc$ is holomorphic and satisfies
 $$
 |v(\zeta)|\leq C\,(1+\|\zeta\|)^N\,e^{R\,\|\text{Im}\,\zeta\|}
 $$
 for all $\zeta\in\cn$ and $v/u$ is holomorphic on $\cn$, then
 $$
 |v(\zeta)/u(\zeta)|\leq  C'\,(1+\|\zeta\|)^{N'}\,e^{R'\,\|\text{Im}\,\zeta\|}
 $$
 for all $\zeta\in\cn$.
 \end{proposition}



The slow decrease condition \ref{slow-decrease} implies, in particular, that constant coefficient differential operators on $\mathbb R^n$ are invertible,
since their Fourier transforms are polynomials. Below we shall see  that it also implies that  finite sums of delta functions are invertible.
This result is already obtained in \cite{Ehrenpreis1955} , but we include a proof here
in order to demonstrate the use of Theorem~\ref{T:invertible-dist}\ref{slow-decrease} and \eqref{E:invertibility2}.

\begin{proposition}\label{T:deltafcn}
Fix distinct distinct $x_1,\ldots,x_N$ in $\mathbb R^n$.  Then the distribution $\mu=\sum_{j=1}^N \delta_{x_j}$is invertible.
\end{proposition}
\begin{proof}
Without loss of generality, we may assume that $N>1$ (otherwise $c_\mu$ is just a translation which is trivially surjective) and that $\|x_1\|\geq \|x_j\|$ for all $j$.  Then in particular, $x_1\neq 0$ and $\langle x_j,x_1\rangle<\|x_1\|^2$ for all $j>1$.  Choose any constant $A$ such that
\begin{equation}\label{E:A-est}
A > \frac{\|x_1\|\,\log N}{\log 2(\|x_1\|^2-\langle x_j,x_1\rangle)}\qquad (j=2,\ldots,N).
\end{equation}

For any $\zeta\in \mathbb C^n$, we have $\mu^*(\zeta)=\sum_{j=1}^N e^{-i\langle x_j,\zeta\rangle}$.  Thus if $\xi,\,\eta\in\mathbb R^n$ and $\zeta=\xi+i\eta$, we have
 \begin{align}
 |\mu^*(\zeta)|&=\left|\sum_{j=1}^N e^{\langle x_j,\eta\rangle}\,e^{-i\langle x_j,\xi\rangle}\right|\nonumber\\
 &=e^{\langle x_1,\eta\rangle}\,\left|e^{-i\langle x_1,\xi\rangle}+\sum_{j=2}^N e^{(\langle x_j,\eta\rangle-\langle x_1,\eta\rangle)}\,e^{-i\langle x_j,\xi\rangle}\right|\nonumber\\
 &\geq e^{\langle x_1,\eta\rangle}\,\left(1-\sum_{j=2}^N e^{(\langle x_j,\eta\rangle-\langle x_1,\eta\rangle)}\right)\label{E:fourierestimate1}
 \end{align}
Fixing $\xi$ for the moment, let $\eta=tx_1$, choosing $t$ so that 
 $$
 \frac{\|x_1\|\,\log N}{\|x_1\|^2-\langle x_j,x_1\rangle}<t\|x_1\|\leq A\,\log(2+\|\xi\|)\qquad(j=2,\ldots,N).
 $$
 This is possible, by our choice \eqref{E:A-est} of $A$.  Then $\|\zeta-\xi\|=\|\eta\|=t\|x_1\|\leq A\log(2+\|\xi\|)$, and 
moreover, $t(\langle x_j,x_1\rangle-\|x_1\|^2)<-\log N$ for $j\geq 2$.   Hence
by \eqref{E:fourierestimate1} we have
 \begin{align}
 |\mu^*(\zeta)|&\geq e^{t\|x_1\|^2}\,\left(1-\sum_{j=2}^N e^{t(\langle x_j,x_1\rangle-\|x_1\|^2)}\right)\nonumber\\
 &\geq e^{t\|x_1\|^2}\cdot\frac 1N\label{E:fourierestimate2}
 \end{align}
 If we now choose $t$ so that $t\|x_1\|=A\,\log(2+\|\xi\|)$, then we see that the Fourier estimate \eqref{E:fourierestimate2} becomes
 $$
 |\mu^*(\zeta)|\geq \frac 1N (2+\|\xi\|)^{A\|x_1\|},
 $$
 which will certainly imply the slow decrease condition \eqref{E:invertibility2}.
\end{proof}

\noindent\emph{Remark.} For distinct points $x_1,\ldots,x_N$ in $\mathbb R^n$ and nonzero complex scalars $c_1,\ldots,c_N$, the above proof can be easily modified to show that the distribution
$$
\mu=\sum_{j=1}^N c_j\delta_{x_j}
$$
is invertible.  In this case, we can again assume that $N>1$ and that $\|x_1\|$ is maximal, and we choose $A$ so that
$$
A>\frac{\|x_1\|(\log M+\log |c_j|-\log |c_1|)}{\log2(\|x_1\|^2-\langle x_j,x_1\rangle)}\qquad (j=2,\ldots,N),
$$
where $M$ is any constant such that $M>N$ and $M>2|c_j|$ for $j=1,\ldots,N$.  We will leave the details to the reader.

We finish this section by including the more general result from \cite[Theorem 5]{Ehrenpreis1955} which we will use later
\begin{theorem}\label{T:delaydiffinv}
Fix distinct points $x_1,\ldots,x_N$ in $\mathbb R^n$ and let $p_1,\ldots,p_N$ be polynomials in $\rn$.  
Then the distribution $\mu=\sum_{j=1}^N p_j(\partial_1,\dots,\partial_n)\delta_{x_j}$ is invertible.
\end{theorem}

\section{Noncompact Symmetric Spaces: Preliminaries and Notation}

Now let $X=G/K$ be a noncompact symmetric space, where $G$ is a real noncompact semisimple Lie group with finite center, 
and $K$ is a maximal compact subgroup.  We will now fix the more or less standard terminology associated with these spaces, which may be found, for example, in Helgason's books \cite{DS} and \cite{GGA}.

Let $\mathfrak g$ denote the Lie algebra of $G$ and denote by $\langle X,Y\rangle$ the Killing form
$$
\langle X,Y\rangle = \mathrm{Tr} (\mathrm{ad}(X) \mathrm{ad}(Y)).
$$
If $\mathfrak k$ denotes the Lie algebra of $K$,
then we have a Cartan decomposition
$$
\mathfrak g=\mathfrak k \oplus\mathfrak p,
$$
where $\mathfrak p$ is the orthogonal complement of $\mathfrak k$ under the Killing form on $\mathfrak g$.  
The Killing form restricted to $\mathfrak p$ is positive definite, and we define a norm on 
$\mathfrak p$ by
$$
\| X\| = \sqrt{\langle X,X\rangle}.
$$
We endow the symmetric space $X$ with the left invariant Riemannian metric induced from 
this norm on $\mathfrak p \cong T_o\,(G/K)$, where $o$ is the identity coset $\{K\}$ in $X=G/K$.

Let $\mathfrak a$ be a maximal abelian subspace of $\mathfrak p$, and let $\Sigma$ denote the set of (restricted) roots of $\mathfrak g$ with respect to $\mathfrak a$.  For each $\alpha\in\Sigma$, let $\mathfrak g_\alpha$ be the corresponding root space, $\mathfrak g_\alpha=\{X\in\mathfrak g\colon [H,X]=\alpha(H)\,X\text{ for all }H\in\mathfrak a\}$.   Then we have the root space decomposition
$$
\mathfrak g=\mathfrak g_0\oplus\sum_{\alpha\in\Sigma}\mathfrak g_\alpha,
$$
with $\mathfrak g_0$ the centralizer of $\mathfrak a$ in $\mathfrak g$.  We have $\mathfrak g_0=\mathfrak h_k\oplus\mathfrak a$,
where $\mathfrak h_k$ is the centralizer of $\mathfrak a$ in $\mathfrak k$.

We fix a positive Weyl chamber $\mathfrak a^+$ in $\mathfrak a$, and let $\Sigma^+$ denote the corresponding set of positive restricted roots.  In addition, let $\Sigma_0$ and $\Sigma^+_0$ denote the set of indivisible roots and positive indivisible roots, respectively.  Let
$$
\mathfrak n=\sum_{\alpha\in\Sigma^+} \mathfrak g_\alpha,
$$
and let $N$ be the analytic subgroup of $G$ with Lie algebra $\mathfrak n$.  If $A$ is the analytic subgroup of $G$ with Lie algebra $\mathfrak a$, then we have the Iwasawa decomposition
\begin{equation}\label{E:iwasawa}
G=NAK.
\end{equation}
If $g\in G$, we write $g=n(g)\,\exp A(g)\,k(g)$, in accordance with \eqref{E:iwasawa}, with $A(g)\in\mathfrak a$.

If $\alpha\in\Sigma$, its \emph{multiplicity} is $m_\alpha=\dim\g_\alpha$.  We put
\begin{equation}\label{E:rho}
\rho=\frac 12\,\sum_{\alpha\in\Sigma^+} m_\alpha\,\alpha.
\end{equation}
Let $\{\alpha_1,\ldots,\alpha_l\}$ be the set of all simple roots.  The root lattice $\Lambda$ is the subset of $\a^*$ consisting of all sums $\sum_{j=1}^l k_j\,\alpha_j$, with each $k_j\in\zz$.  We also put
$\Lambda_+=\{\sum_{j=1}^l k_j\alpha_j\,:\,k_j\in\zz^+\text{ for all }j\}$.

For each $g\in G$, let $\tau(g)$ denote the left translation $x\mapsto g\cdot x$ on $X$, and if $f$ is a function on $X$, we put $\tau(g)f(x)=f(g^{-1}\cdot x)$.

Let $M$ be the centralizer and $M'$ the normalizer of $A$ in $K$, and let $W$ be the quotient group $M'/M$.  Then $W$ is the Weyl group associated with the root system $\Sigma$.
Let $\mathbb D(X)$ denote the algebra of left invariant differential operators on $X$, and let $\Gamma$ be the Harish-Chandra isomorphism from $\mathbb D(X)$ to the  algebra $I(\mathfrak a)=S(\mathfrak a)^W$ of $W$-invariant elements of the symmetric algebra $S(\mathfrak a)$.

A \emph{spherical function} on $X$ is a $K$-invariant joint eigenfunction $\varphi$ of $\mathbb D(X)$ normalized so that $\varphi(o)=1$.  The spherical functions are parametrized by the orbit space $\mathfrak a^*_\mathbb C/W$, where $\mathfrak a^*_\mathbb C$ is the complexified dual space of $\mathfrak a$.  The spherical function corresponding to $\lambda\in\mathfrak a^*_\mathbb C$ (or rather its $W$ orbit) is
\begin{equation}\label{E:spherical1}
\varphi_\lambda(x)=\int_B e^{(i\lambda+\rho)A(x,b)}\,db\qquad(x\in X).
\end{equation}
Here $B=K/M$, $db$ is the normalized $K$-invariant measure on the coset space $B$, and if $x=g\cdot o,\;b=kM$, we have put $A(x,b)=A(k^{-1}g)$.  $A(x,b)$ represents the ``directed distance'' from $o$ to the horocycle passing through $x$ and with ``normal'' $b$.  More precisely, by the Iwasawa decomposition, if $b=kM$, then we have $x=kan\cdot o$ for unique $a\in A$  (independent of the choice of  representative in the coset $kM$) and $n\in N$,  and $a=\exp A(x,b)$.

The spherical function $\varphi_\lambda$ satisfies
\begin{equation}\label{E:spherfcn-invariance}
\varphi_{w\cdot\lambda}(x)=\varphi_\lambda(x),\qquad \varphi_\lambda(k\cdot x)=\varphi_\lambda(x)
\end{equation}
for all $x\in X,k\in K,\,\lambda\in\a^*_{\mathbb C}$, and $w\in W$.

Later we will need the following integration formula based on the Iwasawa decomposition \eqref{E:iwasawa}.
Using the fact that $A$ normalizes $N$, we have $G=ANK$, so with appropriate normalizations of the Haar measures
on $A$ and $N$ we have
\begin{equation}\label{E:iwasawa2}
\int_X f(x)\,dx = \int_A \int_{N} f(an\cdot o)\,dn\,da,
\end{equation}
when this integral makes sense.

\section{The Fourier and Radon Transforms on $G/K$}
The \emph{Fourier transform} of a function $f\in\mathcal D(X)$ is the function $\widetilde f$ on $\a^*_\mathbb C\times B$ given by
\begin{equation}\label{E:ft-function}
\widetilde f(\lambda,b)=\int_X f(x)\,e^{(-i\lambda+\rho)\,A(x,b)}\,dx \qquad(\lambda\in\a^*_\mathbb C,\;b\in B).
\end{equation}
The Fourier transform extends naturally to compactly supported distributions on $X$: if $S\in\mathcal E'(X)$, $\widetilde S$ is the function on $\a^*_\mathbb C\times B$ given by
\begin{equation*}
  \widetilde S(\lambda,b) = S(e_{-\lambda,b})
\end{equation*}
where $e_{\lambda,b}(x) = e^{(i\lambda+\rho)\,A(x,b)}$ which is in $\mathcal E(X)$.
This definition is often written as an integral
\begin{equation}\label{E:ft-dist}
\widetilde S(\lambda,b)=\int_Xe^{(-i\lambda+\rho)\,A(x,b)}\,dS(x).
\end{equation}
In case $f$ (or $S$) is $K$-invariant, the Fourier transform becomes the \emph{spherical Fourier transform}, which for $S$ is given by
\begin{equation}\label{E:spherft}
\widetilde S(\lambda)=S(\varphi_{-\lambda}) = \int_X \,\varphi_{-\lambda}(x)\,dS(x) \qquad (\lambda\in \a^*_{\mathbb C}.)
\end{equation}
Moreover, if $S,\,T\in\mathcal E'(X)$, and if $T$ is $K$-invariant, we have
\begin{equation}\label{E:ft-conv}
(S*T)^\sim(\lambda,b)=\widetilde S(\lambda,b)\,\widetilde T(\lambda) \qquad (\lambda\in \a^*_{\mathbb C}, \;b\in K/M),
\end{equation}
with analogous relations in case $S$ or $T$ are replaced by elements of $\mathcal D(X)$.

The Fourier transform on $X$ is intimately connected to the \emph{horocycle Radon transform}.  A \emph{horocycle}  is an orbit in $X$ of a conjugate of $N$.   If  $\Xi$ denotes the set of all horocycles, then  $G$ acts transitively on $\Xi$ and the isotropy subgroup of the ``fundamental'' horocycyle $\xi_0=N\cdot o$ is $MN$.  Thus we can identify $\Xi$ with the homogeneous manifold $G/MN$.  Moreover, the Iwasawa decomposition also shows that the map
\begin{align}
K/M\times A&\to \Xi\nonumber\\
(kM,a)&\mapsto ka\cdot\xi_0\label{E:polar}
\end{align}
is a diffeomorphism.  The \emph{horocycle Radon transform} maps suitable functions on $X$ to suitable functions on $\Xi$, and is given by
\begin{equation}\label{E:horocycle-radon}
\widehat f(ka\cdot\xi_0)=\int_N f(kan\cdot o)\,dn\qquad (k\in K,\,a\in A).
\end{equation}
In particular, the map $f\mapsto\widehat f$ is a continuous linear map from $\mathcal D(X)$ into $\mathcal D(\Xi)$.  (See \cite{GASS}, Ch. I, \S3 for general continuity properties of integral transforms on homogeneous spaces.)  The Iwasawa decomposition implies the following  ``projection-slice'' relation between the Fourier and Radon transforms:
\begin{align}\label{E:proj-slice}
\widetilde f(\lambda,b)&=\int_A \widehat f(b,a)\,e^{(-i\lambda+\rho)\log a}\,da\nonumber\\
&=(e^{\rho(\log (\cdot))}\,\widehat f(b,\cdot))^*(\lambda)\qquad ((\lambda,b)\in\a^*_\cc\times B).
\end{align}

The \emph{dual horocycle transform} $\psi\mapsto\widecheck\psi$ maps $\mathcal E(\Xi)$ to $\mathcal E(X)$ by integrating over horocycles containing a given point:
\begin{equation}\label{E:dual-hor}
\widecheck\psi(g\cdot o)=\int_K \psi(gk\cdot\xi_0)\,dk\qquad (g\in G)
\end{equation}
for $\psi\in\mathcal E(X)$.   This transform is a continuous map from $\mathcal E(\Xi)$ to $\mathcal E(X)$ (again see \cite{GASS}, Ch. I, \S3).  Our group-theoretic setup implies that the horocycle transform and its dual are formal adjoints, hence the term \emph{dual transform}:
\begin{equation}\label{E:duality}
\int_{B\times A} \widehat f(b,a)\,\psi (b,a)\,e^{2\rho(\log a)}\,da\,db=\int_X f(x)\,\widecheck\psi(x)\,dx,
\end{equation}
for $f\in\mathcal D(X)$ and $\psi\in\mathcal E(\Xi)$.  (The $G$-invariant measure on $\Xi=G/MN=K/M\times A$ is $e^{2\rho(\log a)}\,da\,db$.)

We can use the adjoint relation \eqref{E:dual-hor} to define the Radon transform of any compactly supported distribution on $X$.  If $S\in\mathcal E'(X)$, its Radon transform $\widehat S$ is the distribution on $\Xi$ given by
\begin{equation}\label{E:radon-dist}
\widehat S(\psi)=S(\widecheck\psi)\qquad (\psi\in\mathcal E(\Xi)).
\end{equation}
Being the adjoint of  the map $\psi\mapsto\widecheck\psi$, we see that the map $S\mapsto\widehat S$ is a continuous linear map
from $\mathcal E'(X)$ to $\mathcal E'(\Xi)$.  

To derive a projection-slice theorem for distributions, we first define  ``restriction'' maps
$S\mapsto\widehat S_b$ from $\mathcal E'(X)$ to $\mathcal E'(A)$ for each $b\in B$ as follows.   Fix $b\in B$.  If $F\in\mathcal E(A)$, consider the  function $F^b\in\mathcal E(X)$ given by
\begin{equation}\label{E:plane-wave}
F^b(x)=F(\exp A(x,b))\qquad (x\in X).
\end{equation}
Let $b=kM$.  Then the function $F^b$ is  constant on horocycles with normal $b$; that is, on the horocycles $ka\cdot\xi_0$, and so is a \emph{horocycle plane wave}.  Since $X=kAN\cdot o$, it is not hard to see that $F\mapsto F^b$ is a continuous linear map from $\mathcal E(A)$ to $\mathcal E(X)$.

We define the map $S\mapsto\widehat S_b$ to be the adjoint of the map $F\mapsto F^b$:
\begin{equation}\label{E:restriction}
\widehat S_b(F)=S(F^b)\qquad (F\in\mathcal E(A)).
\end{equation}
The map $S\mapsto \widehat S_b$ is therefore a continuous linear map from $\mathcal E'(X)$ to $\mathcal E'(A)$.  In case $S=f\in\mathcal D(X)$, then by \eqref{E:iwasawa2}
$$
\widehat f_b(a)=\widehat f(b,a)=\int_N f(kan\cdot o)\,dn\qquad (a\in A, b=kM).
$$
From \eqref{E:ft-dist}, we now obtain the  projection-slice theorem for distributions:
\begin{equation}\label{E:proj-slice-dist}
\widetilde S(\lambda,b)=\left(e^\rho\,\widehat S_b\right)^*(\lambda)
\end{equation}
where the Euclidean Fourier transform is taken over $\mathfrak a$.
If $\mu\in\mathcal E'(X)$ is $K$-invariant, relation \eqref{E:ft-conv} shows that
\begin{equation}\label{E:rt-conv}
e^\rho\,\widehat{(S*\mu)}_b=(e^\rho\,\widehat S_b)*\mu_\mathcal{A},
\end{equation}
where the convolution on the right hand side is taken over the Euclidean space $\a$, and $\mu_\mathcal{A}\in\mathcal E'(\a)$ is the \emph{Abel transform} of $\mu$:
\begin{equation}\label{E:abel}
\mu_\mathcal{A}=e^\rho\,\widehat\mu_b.
\end{equation}
Here $b$ is any element of $B$.  (Since $\mu$ is $K$-invariant, the choice of $b$ does not matter.)  

Note also that as a special case of \eqref{E:proj-slice-dist}, we have
\begin{equation}\label{E:abel-ft}
(\mu_a)^*(\lambda)=\widetilde\mu(\lambda)\qquad (\lambda\in\a^*_\cc).
\end{equation}

There are Paley-Wiener theorems that describe the ranges of the Fourier transforms \eqref{E:ft-function} and \eqref{E:ft-dist}.
To state them properly,  we first note that $\widetilde S(\lambda,b)$ and $\widetilde f(\lambda,b)$ are smooth on $\a^*_\mathbb C\times B$ and holomorphic in $\lambda$; moreover, the function on $\a^*_\mathbb C$ given by
\begin{equation}\label{E:ft-W-inv}
\lambda\mapsto\int_B\widetilde S(\lambda,b)\,e^{(i\lambda+\rho)A(x,b)}\,db
\end{equation}
turns out to be $W$-invariant, and $\widetilde f$ satisfies a similar property.  The relation \eqref{E:ft-W-inv} is a consequence of the functional relation
$$
\varphi_\lambda(g^{-1}h\cdot o)=\int_B e^{(-i\lambda+\rho)A(g\cdot o,b)}\,e^{(i\lambda+\rho)A(h\cdot o,b)}\,db, 
$$
for all $\lambda\in\a^*_\cc$, and $g,\,h\in G$.

The space $\mathfrak a$ is equipped with the Killing form inner product from $\mathfrak p$. Fix a basis for $\mathfrak a$ and a dual basis on $\mathfrak a^*$
arising from the Killing form inner product. Denote by $\langle \lambda,\eta \rangle$ the Euclidean inner product of $\lambda,\eta \in \mathfrak a^*$ in terms of
the basis for $\mathfrak a^*$. Finally extend this inner product to a bilinear form on $\mathfrak a^*_\mathbb{C} = \mathfrak a^*+i\mathfrak a$, and let $\| \lambda \|$
denote the norm of $\lambda \in \mathfrak a^*_\mathbb{C}$ inherited from the bilinear form.

For $R>0$, a function $\psi(\lambda,b)$, smooth on $\a^*_\cc\times B$ and holomorphic in $\lambda$ is said to be \emph{rapidly decreasing of uniform exponential type $R$} provided that for any $N\in\zz^+$, $\psi$
 satisfies the condition 
\begin{equation}\label{E:exptype3}
\sup_{(\lambda,b)\in\a^*_\cc\times B}(1+\|\lambda\|)^N\,e^{-R\,\|\text{Im}\,\lambda\|}\,|\psi(\lambda,b)|<\infty.
\end{equation}
Let $\mathcal H_R(\a^*_\cc\times B)$ denote the vector space of all such functions, and let $\mathcal H(\a^*_\cc\times B)$ be their union for all $R$.  Let $\mathcal H_R^W(\a^*_\cc\times B)$ denote the subspace of $\mathcal H_R(\a^*_\cc\times B)$ consisting of functions $\psi$ satisfying the invariance condition \eqref{E:ft-W-inv}, and let $\mathcal H^W(\a^*_\cc\times B)$ be their union.

Let $\mathcal D_R(X)$ denote the vector space of all $C^\infty$ functions on $X$ with support in the closed ball $\overline B_R(o)$.
Then we have the following Paley-Wiener theorem.
\begin{theorem}\label{T:PW1} (\cite{He1}, Theorem 8.3.)
  The Fourier transform $f\mapsto\widetilde f$ is a linear bijection from $\mathcal D_R(X)$ onto $\mathcal H^W_R(\a^*_\cc\times B)$.
\end{theorem}

We now state the corresponding Paley-Wiener theorem for compactly supported distributions.  For this, we say that a function $\Psi(\lambda,b)$, smooth on $\a^*_\cc\times B$ and holomorphic in $\lambda$, is \emph{ of uniform exponential type $R>0$ in $\a^*_\cc$ and of slow growth} provided that there exist constants $A>0$ and $N\in\zz^+$ such that
\begin{equation}\label{E:exptype4}
|\Psi(\lambda,b)|\leq A\,(1+\|\lambda\|)^N\,e^{R\,\|\text{Im}\,\lambda\|}\qquad ((\lambda,b)\in\a^*_\cc\times B).
\end{equation}
Let $\mathcal K_R(\a^*_\cc\times B)$ denote the vector space of all such functions, and let $\mathcal K_R^W(\a^*_\cc\times B)$ be the subspace consisting of those functions $\Psi$ satisfying the $W$-invariance condition \eqref{E:ft-W-inv}.  Finally, let $\mathcal K(\a^*_\cc\times B)$ and $\mathcal K^W(\a^*_\cc\times B)$ denote the union of the subspaces $\mathcal K_R$ and $\mathcal K_R^W$, respectively, for all $R>0$.

Let $\mathcal E'_R(X)$ denote the subspace of $\mathcal E'(X)$ consisting of all distributions with support in $\overline B_R(o)$.

\begin{theorem}\label{T:PW2} (\cite{GASS}, Ch. III, Corollary 5.9.)The Fourier transform $S\mapsto\widetilde S$ is a linear bijection from $\mathcal E'_R(X)$ onto $\mathcal K^W_R(\a^*_\cc\times B)$.
\end{theorem}

For proofs of the two Paley-Wiener theorems above, see \cite{GASS}, Ch. III, \S5.

\section{A Template for Surjectivity}
Suppose that $\mu\in\mathcal E'(X)$ is $K$-invariant.  Let $c_\mu$ be the convolution operator  on $\mathcal E(X)$ given by
$c_\mu(f)=f*\mu$.  Any orthonormal basis of $\mathfrak a$ provides a linear isometry from
$\mathfrak a_\cc^*$ onto $\cc^l$ where
$l=\dim(\mathfrak a)$. A function $u$ on $\a^*_\cc$ is called \emph{slowly decreasing} if it is slowly decreasing as a function on $\cc^l$.
Our aim in this section is to prove the following theorem.

\begin{theorem}\label{T:surjectivity}  Let $\mu$ be a $K$-invariant distribution in $\mathcal E'(X)$ whose spherical Fourier transform $\widetilde\mu(\lambda)$ is a slowly decreasing function on $\a^*_\cc$.  Then the convolution operator $c_\mu\colon \mathcal E(X)\to\mathcal E(X)$ is
surjective.
\end{theorem}

Note that since $\mu$ is $K$-invariant, the $W$-invariance \eqref{E:spherfcn-invariance} of $\lambda\mapsto\varphi_\lambda(x)$ and
\eqref{E:spherft} imply that
 $\widetilde \mu$ is  $W$-invariant, and the forward Paley-Wiener Theorem (a consequence of the projection-slice theorem \eqref{E:proj-slice-dist}) shows that $\widetilde \mu(\lambda)$ is   of  exponential type and slow growth in $\a^*_\cc$.

\begin{proof}
Suppose that $\widetilde\mu(\lambda)$ is slowly decreasing.  We now define the $K$-invariant distribution $\mu^\vee\in\mathcal E'(X)$ as follows. Noting that $\mu$ is determined by its restriction to the (closed) subspace $\mathcal E^\#(X)$ of $\mathcal E(X)$ consisting of all $K$-invariant functions, we put
$$
\mu^\vee(f)=\int_{G/K} f(g^{-1}K)\,d\mu(gK)\qquad\qquad(f\in\mathcal E^\#(X)).
$$
Note that the function $gK\mapsto f(g^{-1}K)$ belongs to $\mathcal E^\#(X)$.

Now the adjoint map to $c_\mu\colon\mathcal E(X)\to\mathcal E(X)$ is
\begin{align}\label{E:c-adjoint}
c_{\mu^\vee}\colon\mathcal E'(X)&\to\mathcal E'(X)\nonumber\\
T&\mapsto T*\mu^\vee.
\end{align}
We will show that this latter map is injective and has closed range in the strong (and hence weak$^*$) topology on $\mathcal E'(X)$.
The theorem will then follow from Theorem \ref{T:frechet}.  (As mentioned earlier, strongly closed subspaces of $\mathcal E'(X)$ are also weak$^*$ closed.)

The spherical Fourier transform of $\mu^\vee$ is $\widetilde\mu(-\lambda)$, which is also slowly decreasing, so to simplify the notation, we will replace 
$\mu^\vee$ by $\mu$ and show that the map $c_\mu\colon\mathcal E'(X)\to\mathcal E'(X)$ is injective and has closed range.
Let us first show that the map $c_\mu$ is injective. 
The set of all $\lambda$ for which $\widetilde\mu(\lambda)\neq 0$ is open and dense in $\a^*_\cc$, and therefore
$T*\mu=0$ implies that
$$
\widetilde T(\lambda,b)\,\widetilde\mu(\lambda)=0\qquad ((\lambda,b)\in\a^*_\cc\times B),
$$
from which we obtain $\widetilde T(\lambda,b)\equiv 0$, and hence $T=0$.




Now we claim that
\begin{multline}\label{E:range1}
c_\mu(\mathcal E'(X))\\
=\{T\in\mathcal E'(X)\,:\,\widetilde T(\lambda,b)/\widetilde\mu(\lambda)\text{ is holomorphic in }\lambda\text{ for each }b\in B\}
\end{multline}
and that this set is closed in $\mathcal E'(X)$. 

From relation \eqref{E:ft-conv}, it is clear that the left hand side above is contained in the right.  On the other hand, suppose that $T$ belongs to the right hand side above. Since 
$\widetilde T(\lambda,b)/\widetilde\mu(\lambda)$ is holomorphic for each fixed $b$,
Proposition \ref{T:smoothness} in Appendix B then implies that $\widetilde T(\lambda,b)/\widetilde\mu(\lambda)$ is smooth on $\a^*_\cc\times B$.

By the forward Paley-Wiener Theorem on $X$, there exist positive constants $A$ and $R$  and an integer $N\in\zz^+$ 
(all of which do not depend on $b$)
such that $\widetilde T$ satisfies the growth condition
\begin{equation}\label{E:ft-exptype}
|\widetilde T(\lambda,b)|\leq A\,(1+\|\lambda\|)^N\,e^{R\,\|\text{Im}\,\lambda\|}\qquad (\lambda,\in\a^*_\cc)
\end{equation}
for all $b\in B$.
Hence by Proposition \ref{T:uniform-exp-type}, there exist positive constants $A'$ and $R'$, and an integer $N'\in\zz^+$, such that
\begin{equation}\label{E:ft-exptype1}
\left|\frac{\widetilde T(\lambda,b)}{\widetilde \mu(\lambda)}\right|\leq A'\,(1+\|\lambda\|)^{N'}\,e^{R'\,\|\text{Im}\,\lambda\|}\qquad (\lambda,\in\a^*_\cc)
\end{equation}
for all $b\in B$.  It follows that $\widetilde T(\lambda,b)/\widetilde \mu(\lambda)\in\mathcal K(\a^*_\cc\times B)$,
and it remains to be shown that it is in $K^W(\a^*_\cc\times B)$.

By assumption $\widetilde T\in\mathcal K^W(\a^*_\cc\times B)$ 
which means that
$$
\int_B \widetilde T(\lambda,b)\,e^{(i\lambda+\rho)\,A(x,b)}\,db=\int_B \widetilde T(\sigma\lambda,b)\,e^{(i\sigma\lambda+\rho)\,A(x,b)}\,db.
$$
The $W$-invariance of $\widetilde\mu(\lambda)$ thus gives
$$
\int_B \frac{\widetilde T(\lambda,b)}{\widetilde\mu(\lambda)}\,e^{(i\lambda+\rho)\,A(x,b)}\,db=\int_B \frac{\widetilde T(\sigma\lambda,b)}{\widetilde\mu(\sigma\lambda)}\,e^{(i\sigma\lambda+\rho)\,A(x,b)}\,db
$$
for all $\sigma\in W$ and all $\lambda\in\a^*_\cc$ for which $\widetilde\mu(\lambda)\neq 0$.

This set is dense in $\a^*_\cc$, and the relation above
therefore holds for all $\lambda$ by continuity, since the integrands are uniformly continuous on compact sets.

Now that we have established that $\widetilde T(\lambda,b)/\widetilde\mu(\lambda)\in\mathcal K^W(\a^*_\cc\times B)$, the Paley-Wiener Theorem on $X$ (Theorem \ref{T:PW2}) implies that there exists a distribution $S\in\mathcal E'(X)$ such that $\widetilde S(\lambda,b)=\widetilde T(\lambda,b)/\widetilde\mu(\lambda)$.  Hence $T=S*\mu$, proving the range characterization \eqref{E:range1}.

Next let us prove that the right hand side of \eqref{E:range1} is a closed subset of $\mathcal E'(X)$.  Since the Abel transform $\mu_a$ satisfies \eqref{E:abel-ft}, $\mu_a\in\mathcal E'(\a)$ is an invertible distribution on the Euclidean space $\a$.  Hence by Theorem \ref{T:invertible-dist}, the convolution operator $v\mapsto v*\mu_a$ on $\mathcal E'(\a)$ has closed range $c_{\mu_a}(\mathcal E'(\a))$.  Then by relations \eqref{E:range1}, \eqref{E:proj-slice-dist}, and \eqref{E:rt-conv}, we conclude that
\begin{equation}\label{E:ramge2}
c_\mu(\mathcal E'(X))=\{T\in\mathcal E'(X)\,:\,e^\rho\,\widehat{T}_b\in c_{\mu_a}(\mathcal E'(\a))\text{ for all }b\in B\}
\end{equation}
For each $b\in B$, the linear map $T\mapsto e^\rho\,\widehat{T}_b$ from $\mathcal E'(X)$ to $\mathcal E'(\a)$ is continuous.  It follows that $c_\mu(\mathcal E'(X))$ is a closed subspace of $\mathcal E'(X)$, since $\mathcal E'(\mathfrak{a})*\mu_a$ is closed in $\mathcal E'(\mathfrak{a})$.  This also of course proves that $c_{\mu^\vee}(\mathcal E'(X))$ is closed in $\mathcal E'(X)$, finishing the proof of Theorem \ref{T:surjectivity}.
\end{proof}

\begin{corollary} (Helgason, 1973)
Every nonzero $G$-invariant differential operator on $X$ is a surjective map from $\mathcal E(X)$ onto $\mathcal E(X)$.
\end{corollary}
This is one of the main results in \cite{He1}.  Note that if $D\in\mathcal D(X)$, then $Df=f*D\delta_o$.  Now $D\delta_o\in\mathcal E'(X)$ is $K$-invariant, and
$$
(D\delta_o)^\sim(\lambda)=\Gamma(D)(i\lambda).
$$
Since the right hand side is a polynomial in $\lambda$, it is slowly decreasing, so Theorem \ref{T:surjectivity} applies.

\section{Mean Value Operators on Symmetric Spaces}

Fix a point $y\in X$.  The \emph{mean value operator} $M^y$ is defined on suitable functions $f$ on $X$ by
\begin{equation}\label{E:meanvalue-def}
M^yf(x)=\int_K f(gk\cdot y)\,dk\qquad(x=g\cdot o\in X),
\end{equation}
where $dk$ is the normalized Haar measure on $K$.
If $X$ is of rank one, then the translated orbit $gK\cdot y$ is the sphere in $X$ of  radius $d(o,y)$ (where $d$ denotes the distance in $X$)
and center $g\cdot o$, so the integral in \eqref{E:meanvalue-def} represents the average value of $f$ on this sphere.

Now choose any $g_0\in G$ such that $y=g_0\cdot o$.   Then in terms of the convolution on $X$, we have
\begin{equation}\label{E:MV-convolution}
M^yf=f*\chi_{K\cdot g_0^{-1}\cdot o}\qquad (f\in\mathcal E(X))
\end{equation}
where $\chi_{K\cdot g_0^{-1}\cdot o}\in\mathcal E'(X)$ is the distribution on $X$ given by
$$
\varphi\mapsto\int_K \varphi(kg_0^{-1}\cdot o)\,dk\qquad (\varphi\in\mathcal E(X)).
$$
This distribution is $K$-invariant and is clearly independent of the choice of $g_0$.


Note that for $h\in A$ 
we have
\begin{equation}\label{E:mv-spherft}
\left(\chi_{K\cdot h^{-1}\cdot o}\right)^\sim(\lambda)=\varphi_\lambda (h)\qquad (\lambda\in \a^*_\cc).
\end{equation}
Therefore by Theorem \ref{T:surjectivity}, we see that  
\begin{proposition}
\label{T:slowgrowthspher}
Let $h\in A$ be fixed. If 
the function $\lambda\mapsto \varphi_\lambda(h)$  is slowly decreasing on $\a_\cc^*$, then
$M^h\colon\mathcal E(X)\to \mathcal E(X)$ is surjective. 
\end{proposition}

\section{The Case of Complex $G$}
When $G$ is complex, then $K$ is a compact real form of $G$ and $\mathfrak h=\a+i\a$ is a Cartan subalgebra of $\g$. Let $\Delta$ be the set of roots of $\g$ with respect to $\h$, let $\Delta^+$ be a fixed choice of positive roots, and let $\rho=\sum_{\alpha\in\Delta^+}\alpha$. Let $W$ be the Weyl group corresponding to $\Delta$.

Then for any $h\in A$, we have
\begin{equation}\label{E:complex-sphericalfcn}
\varphi_\lambda(h)=c\;\frac{\pi(\rho)}{\pi(i\lambda)}\;\frac{\sum_{s\in W} \det (s)\,e^{is\lambda(H)}}{\sum_{s\in W} \det (s)\, e^{s\rho(H)}},
\end{equation}
where $\pi(\lambda)=\prod_{\alpha\in\Delta^+}\alpha(\lambda)$. (See \cite{GGA}, Ch. IV, Theorem 5.7.)

Our objective is to show that, for fixed $h\in A$, the holomorphic function $\lambda\mapsto\varphi_\lambda(h)$ is slowly decreasing on $\a^*_\cc$. 

Now $h=\exp H$ for a unique $H\in\a$. We first consider the case when $H\in \a$ is \emph{regular}; that is to say, $\alpha(H)\neq 0$ for all $\alpha\in\Sigma$.  Then $sH\neq s'H$ for all $s\neq s'$ in $W$, and the denominator in \eqref{E:complex-sphericalfcn} does not vanish.  By the remark after Proposition
\ref{T:deltafcn}, the distribution on $\a$ given by $T=\sum_{s\in W} (\det s)\,\delta_{-sH}$ is invertible, so its Fourier transform
$$
T^*(\lambda)=\sum_{s\in W} \det s\,e^{is\lambda(H)}
$$
is slowly decreasing.  It is also divisible by the polynomial $\pi(\lambda)$ in the algebra $\mathcal H(\a^*_\cc)$.   Hence by Condition (iii) of Theorem \ref{T:invertible-dist}, the function 
$$
\varphi_\lambda(\exp H)=\frac{\pi(\rho)}{\sum_{s\in W}(\det s)\,e^{s\rho(H)}}\cdot \frac{T^*(\lambda)}{\pi(\lambda)}
$$
is slowly decreasing.  Theorem \ref{T:surjectivity} now implies that if $\mu$ is the $K$-invariant distribution $\chi_{K\cdot h^{-1}\cdot o}$ on $X$, then $c_\mu : \mathcal E(X)\to\mathcal E(X)$ is surjective.

Suppose now that $H$ is not regular. The function $\lambda\mapsto\varphi_\lambda(h)$ on $\mathfrak a^*_{\mathbb C}$ is of course still holomorphic of exponential type, but the formula \eqref{E:complex-sphericalfcn} for $\varphi_\lambda(h)$ needs to modified since in the present case the ``Weyl denominator'' $\sum_{s\in W} \det(s)\,e^{s\rho(H)}$ equals $0$.

This Weyl denominator can also be written
\begin{equation}\label{E:Weyl-denom}
\sum_{s\in W} \det(s)\,e^{s\rho(H)}=\prod_{\alpha\in{\Delta^+}} \left(e^{\alpha(H)}-e^{-\alpha(H)}\right).
\end{equation}

(See Lemma 24.3 in  \cite{Humphreys1978}.) Now let $\Delta_0$ denote the root system $\{\alpha\in\Delta\,|\,\alpha(H)=0\}$, and let $\Delta_0^+=\Delta_0\cap\Delta^+$. The Weyl group of $\Delta_0$ is the subgroup $W_0$ of $W$ consisting of all elements which leave $H$ fixed, and is generated by the reflections along the root hyperplanes $\pi_\alpha=\alpha^\perp$, where $\alpha\in\Delta_0$ (or even the \emph{simple} root hyperplanes in $\Delta_0$.) Let $\rho_0=(1/2)
\sum_{\alpha\in\Delta_0^+}\alpha$.  The Weyl denominator corresponding to $\Delta_0$ is
\begin{equation}\label{E:W-denom}
\sum_{s\in W_0} \det (s)\,e^{s\rho_0 (H)}=\prod_{\alpha\in\Delta_0^+}\left(e^{\alpha (H)}-e^{-\alpha(H)}\right).
\end{equation}

Let $\pi_0$ denote the polynomial on $\a^*_\cc$ given by  $\pi_0(\lambda)=\prod_{\alpha\in \Delta_0^+}\alpha(\lambda)$ and let $|W_0|$ denote the order of $W_0$. To obtain an explicit expression for $\varphi_\lambda(\exp H)$, we first calculate $\varphi_\lambda(\exp(H+tH_{\rho_0}))$, and note that for small positive $t$, the vector $H+tH_{\rho_0}$ is regular.   Then the fraction on the right hand side of \eqref{E:complex-sphericalfcn} is
\begin{multline*}
\frac{\sum_{s\in W} \det (s)e^{is\lambda(H+tH_{\rho_0})}}{\prod_{\alpha\in\Delta^+}\left(e^{\alpha(H+tH_{\rho_0})}-e^{-\alpha(H+tH_{\rho_0})}\right)}\\=\frac 1{|W_0|}\,
\frac{\sum_{\sigma\in W_0}\sum_{s\in W} \det (\sigma^{-1} s)e^{i\sigma^{-1} s\lambda(H+tH_{\rho_0})}}{\prod_{\alpha\in\Delta^+}\left(e^{\alpha(H+tH_{\rho_0})}-e^{-\alpha(H+tH_{\rho_0})}\right)}
\end{multline*}
The right hand side above can be written
$$
\frac 1{|W_0|}\,\frac{\sum_{s\in W} \det (s)e^{is\lambda(H)}\sum_{\sigma\in W_0}\det (\sigma) e^{is\lambda(t\sigma H_{\rho_0})}}
{\prod_{\alpha\in\Delta^+\setminus\Delta_0^+}\left(e^{\alpha(H+tH_{\rho_0})}-e^{-\alpha(H+tH_{\rho_0})}\right)\prod_{\alpha\in\Delta_0^+}\left(e^{\alpha(t H_{\rho_0})}-e^{-\alpha(tH_{\rho_0})}\right) },
$$
which by \eqref{E:W-denom} equals
$$
\frac 1{|W_0|}\,\frac{\sum_{s\in W} \det (s)e^{is\lambda(H)}\prod_{\alpha\in \Delta_0^+}\left(e^{\alpha(it\,s\lambda)}-e^{-\alpha(it\,s\lambda)}\right)}
{\prod_{\alpha\in\Delta^+\setminus\Delta_0^+}\left(e^{\alpha(H+tH_{\rho_0})}-e^{-\alpha(H+tH_{\rho_0})}\right)\prod_{\alpha\in\Delta_0^+}\left(e^{\alpha(t H_{\rho_0})}-e^{-\alpha(tH_{\rho_0})}\right) },
$$
Taking the limit as $t\to 0$, we obtain
\begin{equation}\label{E:complex-sphericalfcn-nonreg}
\varphi_\lambda(\exp H)=\frac{\pi(\rho)}{|W_0|\,\pi(i\lambda)\,\pi_0(\rho_0)}\,
\frac{ \sum_{s\in W} \det (s)\,\pi_0(is\lambda)\,e^{is\lambda(H)}}
{\prod_{\alpha\in\Delta^+\setminus\Delta_0^+}\left(e^{\alpha(H)}-e^{-\alpha(H)}\right)}
\end{equation}
The exponential polynomial $\psi(\lambda=\sum_{s\in W} \det (s)\,\pi_0(is\lambda)\,e^{is\lambda(H)}$ is skew in $\lambda$, meaning that $\psi(\sigma\lambda)=\det\sigma\,\psi(\lambda)$ for all $\sigma\in W$. This makes $\psi(\lambda)$ divisible by $\pi(i\lambda)$, so the right hand side of \eqref{E:complex-sphericalfcn-nonreg} represents a holomorphic function of $\lambda$.

\begin{theorem}\label{T:cpx-surjectivity}
Let $X=G/K$, with $G$ complex.  For any $h\in A$, the mean value operator
$$
M^h\colon \mathcal E(X)\to\mathcal E(X)
$$
is surjective. 
\end{theorem}

\section{The Case of a General Noncompact Symmetric Space}

We return to the case of a general noncompact symmetric space $X=G/K$. 
 For $M\geq 0$ let $\a^+_M$ be the subchamber $\{H\in\a\,:\,\alpha(H)>M\text{ for all }\alpha\in\Sigma^+\}$, and let $A^+_M=\exp\a^+_M$.  Our aim in this subsection is to prove the following result.
\begin{theorem}\label{T:sym-surj}
There exists a constant $M>0$ such that the mean value operator $M^h\colon\mathcal E(X)\to\mathcal E(X)$ is surjective for all $h\in A^+_M$.
\end{theorem}
We believe that the theorem will be true for all $h$ in $A$, but we are presently not aware of a proof.

By Theorem \ref{T:surjectivity} we will need to prove that a constant $M$ can be found such that for any $h\in A^+_M$ the holomorphic function $\lambda\mapsto\varphi_\lambda(h)$ on $\a^*_\cc$ is slowly decreasing.  The key tool is Harish-Chandra's spherical function expansion
\begin{equation}\label{E:HC-expansion}
\varphi_\lambda(\exp H)=\sum_{s\in W} c(s\lambda) \Phi_{s\lambda}(H),
\end{equation}
where $\Phi_\lambda(H)$ is the \emph{Harish-Chandra series}
\begin{equation}\label{E:HC-series}
\Phi_\lambda(H)=\sum_{\mu\in\Lambda} \Gamma_\mu(\lambda)\,e^{(is\lambda-\rho-\mu)(H)}.
\end{equation}
(See, for example, \cite{GGA}, Chapter IV, for a derivation and treatment.)  In \eqref{E:HC-series} the coefficients $\Gamma_\mu(\lambda),\,\mu\in\Lambda$ are defined by the recursion formula
\begin{align}\label{E:gamma-recursion}
&\Gamma_0(\lambda)=1\\
&(\langle\mu,\mu\rangle-2i\,\langle\mu,\lambda\rangle)\Gamma_\mu(\lambda)\nonumber\\
&\quad=2\sum_{\alpha\in\Sigma^+} m_\alpha\sum_{\substack{k>1,\\ \mu-2k\alpha\in\Lambda^+}}
(\langle\mu+\rho-2k\alpha,\alpha\rangle-i\langle\lambda,\alpha\rangle)\Gamma_{\mu-2k\alpha}(\lambda)\nonumber
\end{align}
and in the expansion \eqref{E:HC-expansion}, $c(\lambda)$ is Harish-Chandra's $c$ function
\begin{multline}\label{E:c-function}
c(\lambda)\\=c_0\,\prod_{\alpha\in\Sigma_0^+}\frac{2^{-i\langle\lambda,\alpha_0\rangle}\,\Gamma(i\langle\lambda,\alpha_0\rangle)}
{\Gamma\left(\frac 12\left(\frac{m_\alpha}2+m_{2\alpha}+\langle i\lambda,\alpha_0\rangle\right)\right)
\Gamma\left(\frac 12\left(\frac{m_\alpha}2+1+\langle i\lambda,\alpha_0\rangle\right)\right)},
\end{multline}
with $\alpha_0=\alpha/\langle\alpha,\alpha\rangle$ and
$$
c_0=\Gamma((1/2)(m_\alpha+m_{2\alpha}+1))2^{(1/2)m_\alpha+m_{2\alpha}}.
$$
The $c$ function is a meromorphic function on $\a^*_\cc$ with poles in the hyperplanes $i\langle\lambda,\alpha_0\rangle=-m$, for all $\alpha\in\Sigma_0^+$ and $m\in\zz^+$.

The equality \eqref{E:HC-series} holds (and the series $\Phi_{s\lambda}(H)$ converge for all $s\in W$) when $H\in\a^+$ and 
$\lambda\in\a^*_\cc$ satisfy $\langle\mu,\mu\rangle-2i\langle\mu,s\lambda\rangle\neq 0$ for all $\mu\in\Lambda\setminus\{0\}$ and $i(s\lambda-s'\lambda)\notin
\widetilde\Lambda$ for all $s\neq s'$ in  $W$.

In particular, the Harish-Chandra series \eqref{E:HC-series} converges for all $\lambda\in\a^*$.  While the expansion \eqref{E:HC-series} is employed mostly to study the spherical function $\varphi_\lambda(\exp H)$ as a function of $H$ (for fixed $\lambda$), for the purpose of proving Theorem \ref{T:sym-surj} we would like to examine its behavior as $\lambda$ varies, with $H$ fixed, while maintaining the prescribed limitations on $\lambda$.     

Explicitly, in view of the slow decrease criterion \eqref{E:invertibility2} we would like to use the Harish-Chandra expansion to 
show that there exists a  constant $M\geq 0$ such that for all $H\in\a^+_M$, there are positive constants $A,B,C$, and $D$ (depending on $H$) for which
\begin{equation}\label{E:invertibility3}
\sup\{|\varphi_\lambda(\exp H)|\,\colon\,\lambda\in\a^*_\cc,\;\|\lambda-\xi\|\leq A\log(2+\|\xi\|)\}\geq B(C+\|\xi\|)^{-D}
\end{equation}
for all $\xi\in\a^*$.  

We start with a technical lemma which gives an estimate for the coefficients 
$\Gamma_\mu(\lambda)$ when the imaginary part of $\lambda$ is bounded by a given fixed constant.

\begin{lemma}\label{T:gamma-estimate} Suppose that $\eta\in\a^*$ satisfies $\|\eta\|<(1/4)\|\mu\|$ for all $\mu\in\Lambda\setminus \{0\}$.
For any vector $H_0\in\a^+$, there is a constant $K_{H_0}$ such that
\begin{equation}\label{E:gamma-est}
|\Gamma_\mu(\xi+i\eta)|\leq K_{H_0}\,e^{\mu(H_0)}
\end{equation}
for all $\mu\in\Lambda$ and all $\xi\in\a^*$.
\end{lemma}

Since  $\Lambda\setminus\{0\}$ has no accumulation point, the set of all such $\eta$ is a nonempty open ball in $\a^*$.

\begin{proof}
For any $\xi\in\a^*$ and any $\mu\in\Lambda\setminus \{0\}$, our condition for $\eta$ implies that
\begin{align}
|\langle\mu,\mu\rangle-2i\,\langle\mu,\xi+i\eta\rangle|&=|\langle\mu,\mu\rangle+2\,\langle\mu,\eta\rangle-2i\langle\mu,\xi\rangle|\nonumber\\
&\geq \langle\mu,\mu\rangle-2\|\mu\|\,\|\eta\|\nonumber\\
&> \frac 12\langle\mu,\mu\rangle.\label{E:gamma-estimate1}
\end{align}
Hence $\Gamma_\mu(\xi+i\eta)$ is well-defined for each $\mu\in\Lambda$.  

Now recall that in \cite{Gang} the \emph{radial density function} on $\a^+$ is given by
$$
\delta(H)=
\prod_{\alpha\in\Sigma^+} (e^{\alpha(H)}-e^{-\alpha(H)})^{m_\alpha}
$$
We then have series expansions
\begin{align*}
\delta^{1/2}(H)&=e^{\rho(H)}\,\sum_{\nu\in\Lambda} b_\nu e^{-\nu(H)},\\
\delta^{-1/2}(H)&=e^{-\rho(H)}\,\sum_{\nu\in\Lambda} c_\nu e^{-\nu(H)},\\
\delta^{-1/2}(H)\,L_\a(\delta^{1/2})(H)&=\sum_{\nu\in\Lambda} d_\nu e^{-\nu(H)},
\end{align*}
for $H\in\a^+$, where the coefficients $b_\nu,\,c_\nu$, and $d_\nu$ all grow polynomially in $\|\nu\|$, and $d_0=\langle\rho,\rho\rangle$.

Gangolli's modification of the Harish-Chandra series is given by
\begin{align}\label{E:gangolli}
\Psi_\lambda(H)&=\delta^{1/2}(H)\,\Phi_\lambda(H)\nonumber\\
&=\sum_{\mu\in\Lambda} A_\mu(\lambda) e^{(i\lambda-\mu)(H)}
\end{align}
The coefficients $A_\mu(\lambda)$ satisfy the recurrence relation
\begin{equation}\label{E:gangolli-recurrence}
(\langle\mu,\mu\rangle-2i\langle\mu,\lambda\rangle)\,A_\mu(\lambda)=\sum_{\substack{\nu\in\Lambda,\,\nu>0\\
\mu-\nu\in\Lambda}} A_{\mu-\nu}(\lambda)\,d_\nu
\end{equation}
Now in Gangolli's paper \cite{Gang}, the inequality \eqref{E:gamma-estimate1} (without the factor $1/2$) was used for $\lambda\in\a^*+i\a^*_+$ to prove that there exists a constant $C_{H_0}$ such that $|A_\mu(\lambda)|\leq C_H\,e^{\mu(H_0)}$ for all $\mu\in\Lambda$; the relation
$$
\Gamma_\mu(\lambda)=\sum_{\substack{\nu\in\Lambda\\
\mu-\nu\in\Lambda}} c_\nu A_{\mu-\nu}(\lambda)
$$
then implies that there exists a constant $D_{H_0}$ such that
$$
|\Gamma_\mu(\lambda)|\leq D_{H_0}\,e^{\mu(H_0)}
$$
for all $\lambda\in\a^*+i\a^*_+$ and all $\mu\in\Lambda$.

Because the inequality \eqref{E:gamma-estimate1} also holds for $\lambda=\xi+i\eta$, the very same proof shows that there is a constant $K_{H_0}$ satisfying the inequality \eqref{E:gamma-est}.  This finishes the proof of the lemma.
\end{proof}

Our aim is to find a lower bound for the supremum in \eqref{E:invertibility3}. This estimate is
accomplished using $\lambda = \xi - i\eta$ for a fixed $\eta\in \mathfrak{a}^*_+$. This restricts the range of the parameter $\lambda$,
but it is sufficient for our purposes.
From now on we will fix an element $\eta\in\a^*$ satisfying the five conditions below. Condition (b) is assumed to hold  for all $\mu\in\Lambda\setminus\{0\}$, and Conditions (c)--(e) are assumed to hold for all $s\in W$, and all $\alpha\in\Sigma_0^+$:
\begin{enumerate}
\item[(a)] $\eta\in \mathfrak a^*_+$;
\item[(b)] $\|\eta\|<(1/4)\,\|\mu\|$;
\item[(c)] $\langle s\eta,\alpha_0\rangle\notin -\mathbb Z^+$;
\item[(d)] $\langle s\eta,\alpha_0\rangle+m_\alpha/2+m_{2\alpha}\notin -2\mathbb Z^+$;
\item[(e)] $\langle s\eta,\alpha_0\rangle+m_\alpha/2+1\notin -2\mathbb Z^+$.
\end{enumerate}
Since $\a^*_+$ is an open cone in $\a^*$ with vertex at $0$, the set of all $\eta$ satisying (a) and (b) is a nonempty open subset of $\a^*$. Conditions (c)--(e) 
are needed in order to apply Stirling's formula for the Gamma function, and they
stipulate that $\eta$  does not belong to a countable set of hyperplanes in $\a^*$. Their union is a set of measure zero in $\a^*$, so there will be elements $\eta\in\a^*$ satisfying (a)--(e).


\begin{lemma}\label{T:c-fcn-estimate}
Let $\eta$ be the fixed element of $\a^*$ chosen above, and let $s\in W$. Then
\begin{equation}\label{E:c-fcn-asymptotics}
|c(s(\xi-i\eta))|\asymp \prod_{\alpha\in\Sigma_0^+}(1+|\langle\xi,\alpha_0\rangle|)^{-(m_\alpha+m_{2\alpha})/2}
\end{equation}
for all $\xi\in \a^*$.
\end{lemma}

The symbol $\asymp$ in  \eqref{E:c-fcn-asymptotics} means that there are positive constants $r_1$ and $r_2$ (which depend on $s$) such that
\begin{multline}
r_1\,\prod_{\alpha\in\Sigma_0^+}(1+|\langle\xi,\alpha_0\rangle|)^{-(m_\alpha+m_{2\alpha})/2}\\
\leq |c(s(\xi-i\eta))|\leq r_2\,\prod_{\alpha\in\Sigma_0^+}(1+|\langle\xi,\alpha_0\rangle|)^{-(m_\alpha+m_{2\alpha})/2}
\end{multline}
for all $\xi\in\a^*$.

\begin{proof}
Using the  identity $\Gamma(2z)=\pi^{-1/2}2^{2z-1}\,\Gamma(z)\Gamma(z+1/2)$ the $c$ function formula \eqref{E:c-function} becomes
$$
c(\lambda)=c_1\,\prod_{\alpha\in\Sigma_0^+} \frac{\Gamma\left(\frac{i\langle \lambda,\alpha_0\rangle}2\right)
\Gamma\left(\frac{i\langle \lambda,\alpha_0\rangle}2+\frac 12\right)}
{\Gamma\left(\frac{m_\alpha}4+\frac{m_{2\alpha}}2+\frac{\langle i\lambda,\alpha_0\rangle}2\right)
\Gamma\left(\frac{m_\alpha}4+\frac 12+\frac{\langle i\lambda,\alpha_0\rangle}2\right)}
$$
for some positive constant $c_1$.  Then putting $\lambda=\xi-i\eta$, we obtain
\begin{multline}\label{E:c-fcn-resolution}
c(s(\xi-i\eta))\\=c_1\,\prod_{\alpha\in\Sigma_0^+} \frac{\Gamma\left(\frac{\langle s\eta,\alpha_0\rangle+i\langle s\xi,\alpha_0\rangle}2\right)
\Gamma\left(\frac{\langle s\eta,\alpha_0\rangle+i\langle s\xi,\alpha_0\rangle}2+\frac 12\right)}
{\Gamma\left(\frac{m_\alpha}4+\frac{m_{2\alpha}}2+\frac{\langle s\eta,\alpha_0\rangle+i\langle s\xi,\alpha_0\rangle}2\right)
\Gamma\left(\frac{m_\alpha}4+\frac 12+\frac{\langle s\eta,\alpha_0\rangle+i\langle s\xi,\alpha_0\rangle}2\right)}
\end{multline}

Conditions ((c)--(e) for $\eta$ ensure that all the Gamma functions on the right hand side of \eqref{E:c-fcn-resolution}  are well-defined for all $\xi\in\a^*$ and all $s\in W$. Since the Gamma function has no zeros, for our fixed $\eta$ we also see that the right hand side above never vanishes for all $s\in W$ and all $\xi\in\a^*$.

Now we use the following asymptotic formula for the ratio of two Gamma functions
\begin{equation}\label{E:gamma-ratio}
\frac{\Gamma(z)}{\Gamma(z+b)}\approx z^{-b}\,\left(1+O(1/z)\right)\text{ as }|z|\to\infty,
\end{equation}
for $b>0$, which is valid as long as $-\pi+\delta<\text{Arg}\,z<\pi-\delta$, for small positive $\delta$.  (See, for example,  Formula 6.1.47 in \cite{AS}.)

For each $\xi\in\a^*$, the real part of the argument in each of the Gamma functions on the right hand side of \eqref{E:c-fcn-resolution} is constant (and not an integer $\leq 0$), so in particular the asymptotic formula \eqref{E:gamma-ratio} implies that
$$
\left|
\frac{\Gamma\left(\frac{\langle s\eta,\alpha_0\rangle+i\langle s\xi,\alpha_0\rangle}2\right)}{\Gamma\left(\frac{m_\alpha}4+\frac{m_{2\alpha}}2+\frac{\langle s\eta,\alpha_0\rangle+i\langle s\xi,\alpha_0\rangle}2\right)}\right|
\asymp \left(1+|\langle s\xi,\alpha_0\rangle|\right)^{-\frac{m_\alpha}4+\frac{m_{2\alpha}}2}
$$
and
$$
\left|\frac{\Gamma\left(\frac{\langle s\eta,\alpha_0\rangle+i\langle s\xi,\alpha_0\rangle}2+\frac 12\right)}
{\Gamma\left(\frac{m_\alpha}4+\frac 12+\frac{\langle s\eta,\alpha_0\rangle+i\langle s\xi,\alpha_0\rangle}2\right)}\right|
\asymp \left(1+|\langle s\xi,\alpha_0\rangle|\right)^{-\frac{m_\alpha}4}
$$
for any $\alpha\in\Sigma_0^+$. Hence \eqref{E:c-fcn-resolution} implies that
\begin{align*}
|c(s(\xi-i\eta))|&\asymp\prod_{\alpha\in\Sigma_0^+} \left(1+|\langle \xi,s^{-1}\alpha_0\rangle|\right)^{-\frac{m_\alpha+m_{2\alpha}}2}\\
&=\prod_{\alpha\in\Sigma_0^+} \left(1+|\langle \xi,\alpha_0\rangle|\right)^{-\frac{m_\alpha+m_{2\alpha}}2},
\end{align*}
proving the asymptotic relation \eqref{E:c-fcn-asymptotics}.
\end{proof}

Let us now prove Theorem \ref{T:sym-surj}. Again we recall that we have fixed the element $\eta\in\a^*$ satisfying Conditions (a)--(e) above. According to Lemma \ref{T:c-fcn-estimate}, there are constants $m_1$ and $m_2$
such that
$$
|c(\xi-i\eta)|\geq m_1\,\prod_{\alpha\in\Sigma_0^+} (1+|\langle\xi,\alpha_0\rangle|)^{-\frac{m_\alpha+m_{2\alpha}}2}
$$
for all $\xi\in\a^*$, and
$$
|c(s(\xi-i\eta))|\leq m_2\,\prod_{\alpha\in\Sigma_0^+} (1+|\langle\xi,\alpha_0\rangle|)^{-\frac{m_\alpha+m_{2\alpha}}2}
$$
for all $\xi\in\a^*$ and all $s$ in $W$.

For the moment let us fix a vector $H_0\in\a^+$. Since $\eta\in\a^*$ satisfies Condition (a), Lemma \ref{T:gamma-estimate} implies that there is a positive constant $K_{H_0}$ for which 
$$
\left|\Gamma_\mu(s(\xi-i\eta))\right|\leq K_{H_0}\,e^{\mu(H_0)}
$$
for all $\mu\in\Lambda$, all $\xi\in\a^*$, and all $s\in W$. Let $M_1$ be any number larger than $|\alpha(H_0)|$ for all $\alpha\in\Sigma_0^+$. Then by \eqref{E:HC-series}, \eqref{E:c-fcn-asymptotics}, and the estimate above, the sum \eqref{E:HC-expansion} representing $\varphi_{\xi-i\eta}(H)$ converges uniformly on the region $(\xi,H)\in \a^*\times\a_{M_1}$.

In particular, for $(\xi,H)\in\a^*\times \a_{M_1}$, we have
\begin{align}
&e^{\rho(H)}\,|\varphi_{\xi-i\eta}(\exp H)|\nonumber\\
&=\left|\sum_{s\in W}c(s(\xi-i\eta)) \sum_{\mu\in\Lambda} \Gamma_\mu(s(\xi-i\eta)) 
e^{s\eta(H)+is\xi(H)-\mu(H)}\right|\nonumber\\
&\geq |c(\xi-i\eta)|\,e^{\eta(H)}-|c(\xi-i\eta)|\sum_{\mu\in\Lambda\setminus\{0\}}|\Gamma_\mu(\xi-i\eta)|\,e^{\eta(H)-\mu(H)}\nonumber\\
&\qquad\qquad-\sum_{s\neq e} |c(s(\xi-i\eta))|\,
\sum_{\mu\in\Lambda}|\Gamma_\mu(s(\xi-i\eta))| e^{s\eta(H)-\mu(H)}\nonumber\\
&\geq e^{\eta(H)}\,\prod_{\alpha\in\Sigma_0^+} (1+|\langle\xi,\alpha_0\rangle|)^{-\frac{m_\alpha+m_{2\alpha}}2}\times\label{E:spher-fcn-estimate}\\
&\qquad\left(m_1- m_2\sum_{\mu\in\Lambda\setminus\{0\}} K_{H_0} e^{\mu(H_0)-\mu(H)} - m_2\sum_{s\neq e} e^{s\eta(H)-\eta(H)}
\sum_{\mu\in\Lambda} K_{H_0}e^{\mu(H_0)-\mu(H)}\right)\nonumber
\end{align}



Let $\alpha_1,\ldots,\alpha_l$ be the simple roots in $\a^*_+$, and let $H_1,\ldots,H_l$ be a dual basis of $\a$. If $H=\sum_{j=1}^l k_jH_j\in\a$, then $H\in\a^+$ if and only if each $k_j>0$ and for any $M>0$,  $H\in\a_M$ if and only if each $k_j>M$.  

Let $^+\a^*$ be the dual cone $\{\lambda\in\a^*\,:\,\lambda(H)>0\text{ for all }H\in\a^+\}$. Then $\lambda\in\; ^+\a^*$ if and only if  $\lambda$ is nonzero and $\lambda=\sum_{j=1}^l m_j\alpha_j$, where each $m_j\geq 0$.
For $\lambda\in\;^+\a^*$ we put $m(\lambda)=\sum_{j=1}^l m_j$.   Since $\eta\in\a^*_+$, we have $\eta-s\eta\in\,  ^+\a^*$ for all $s\neq e$ in $W$. (See, for instance, \cite{DS}, Ch. VII, Theorem 2.12.)

Let $M>M_1$. If $H\in\a_M^+$, then  $\mu(H)>M\,m(\mu)$ and $\eta(H)-s\eta(H)>M\,m(\eta-s\eta)$. Thus the relation
 \eqref{E:spher-fcn-estimate} implies that
\begin{multline}\label{E:2nd-spher-fcn-est}
e^{\rho(H)}\,|\varphi_{\xi-i\eta}(\exp H)|
\geq e^{\eta(H)}\,\prod_{\alpha\in\Sigma_0^+} (1+|\langle\xi,\alpha_0\rangle|)^{-\frac{m_\alpha+m_{2\alpha}}2}\times\\
\left(m_1-m_2K_{H_0}e^{\mu(H_0)}\left(\sum_{\mu\in\Lambda\setminus\{0\}} e^{-M\,m(\mu)}
+\sum_{s\neq e} e^{-M\,m(\eta-s\eta)}\sum_{\mu\in\Lambda} e^{-M\,m(\mu)} \right)\right)
\end{multline}
Since both $\sum_{\mu\in\Lambda\setminus\{0\}} e^{-M\,m(\mu)}$ and  $\sum_{s\neq e} e^{-M\,m(\eta-s\eta)}$ tend to $0$ as $M\to\infty$, the expression
\begin{equation}\label{E:spher-lower-bd}
m_1-m_2K_{H_0}e^{\mu(H_0)}\left(\sum_{\mu\in\Lambda\setminus\{0\}} e^{-M\,m(\mu)}
+\sum_{s\neq e} e^{-M\,m(\eta-s\eta)}\sum_{\mu\in\Lambda} e^{-M\,m(\mu)} \right)
\end{equation}
is positive for all sufficiently large $M$. Choose one such $M$, and denote the expression \eqref{E:spher-lower-bd} by $C_M$. Then relation
\eqref{E:2nd-spher-fcn-est} gives
\begin{align*}
e^{(-\eta+\rho)(H)}\,|\varphi_{\xi-i\eta}(\exp H)|&\geq C_M\,\prod_{\alpha\in\Sigma_0^+} (1+|\langle\xi,\alpha_0\rangle|)^{-\frac{m_\alpha+m_{2\alpha}}2}\\
&\geq C'C_M\,(1+\|\xi\|)^{-\frac{\dim N}2}
\end{align*}
for all $H\in\a_M$ and all $\xi\in\a^*$, where $C'$ is a constant that depends only on $\Sigma$. Since $\|\eta\|$ is fixed, this clearly implies the slow decrease condition \eqref{E:invertibility3} for each $H\in\a_M$, and this in turn proves Theorem \ref{T:sym-surj}.

\section{Surjectivity in the rank one case}

In this section we assume that $X=G/K$ is of rank one.
The purpose of this section is to show that $\lambda\mapsto \varphi_\lambda(h)$ is
slowly decreasing, thereby proving the surjectivity of mean value operators according
to Proposition~\ref{T:slowgrowthspher}.

Let $\alpha$ and $2 \alpha$ be the positive roots
and let $p= m_{\alpha}$, $q= m_{2 \alpha}$,
respectively. In addition let $n= \dim G/K$.
Here we note that in this case $n= p+q+1$.
We define a norm $|| \cdot ||$ on $\a$ by
$$
|| X || := \left( - \frac{1}{2(p+4q)} B(X, \theta X) \right)^{\frac{1}{2} }.
$$
Here $B(\cdot, \cdot)$ and $\theta$ denote the Killing form and the Cartan involution,
respectively.
We take $H \in \a$
such that $\alpha (H) = 1$.

Next, we identify $\a^*$ with $\rr$
and denote by $\varphi_{\lambda}$ the zonal spherical function on $G/K$
corresponding to $\lambda \in \a^* \cong \rr$.
Then we have the following.

\begin{theorem}[Koornwinder \cite{Ko}]\label{Koornwinder-theorem}
Fix $t>0$ and let $h=\mathrm{Exp}(tH)$. Then
\be\label{Koornwinder-Formula}
\begin{aligned}
{}& 
\frac{ \Gamma (\frac{n-1}{2}) \Gamma (\frac{1}{2})  }{ 2^{ \frac{n-1}{2} } \Gamma (\frac{n}{2} ) }
\; (\sinh t)^{n-2} (\cosh t)^{\frac{q}{2} }  \;
\varphi_{\lambda} ( h ) \\
{}& = \int_0^t \; \cos (\lambda s) \;
     ( \cosh t - \cosh s )^{ \frac{n-3}{2} }  \;
        {}_2 F_1 
      \left( 
 1- \frac{q}{2}, \; \frac{q}{2} \; ;  \; \frac{n-1}{2}  \;  ;
 \frac{ \cosh t - \cosh s }{ 2\cosh t }
     \right) \, ds.
\end{aligned}
\ee
\end{theorem}

\begin{remark} For the details of the above theorem, see also Rouvi\`ere \cite{Rouv} p.113.
\end{remark}

\medskip

We fix $t >0$ and put
\be\label{finite-Fourier-integral-1}
I (\lambda) :=
\int_0^t \; \cos (\lambda s) \;
     ( \cosh t - \cosh s )^{ \frac{n-3}{2} }  \;
        {}_2 F_1 
      \left( 
 1- \frac{q}{2}, \; \frac{q}{2} \; ;  \; \frac{n-1}{2}  \;  ;
 \frac{ \cosh t - \cosh s }{ 2\cosh t }
     \right) \, ds.
\ee
We write the holomorphic extension of $I (\lambda)$ 
as $I (\zeta), \; (\zeta \in \mathbb{C})$.
As a direct consequence of Theorem \ref{Koornwinder-theorem},
it suffices to show that the entire function $\mathbb{C} \ni \zeta \mapsto I (\zeta)$
is slowly decreasing.
We will need to split up in two cases, namely when $n$ is either odd or even.

%
%
%
%

\subsection{The case when $n$ is odd}
First, we consider the case when $n$ is an odd number.
We put $n-3 = 2 \ell$ ($\ell \in \zz, \; \ell \geq 0$).

In this case, we note that 
$p= 2 \ell +2, \; q= 0$ and that the hypergeometric function appearing in the integrand
of $I (\lambda)$ is just a constant function. More precisely, we have
$$
       {}_2 F_1 
      \left( 
 1- \frac{q}{2}, \; \frac{q}{2} \; ;  \; \frac{n-1}{2}  \;  ;
 \frac{ \cosh t - \cosh s }{ 2\cosh t }  
      \right)  =1.
$$
For a nonnegative integer $m$, let
\be
\mathcal{I}_m (\lambda) := 
   \int_{-t}^t \; \cos (\lambda s) \;
      ( \cosh t - \cosh s )^m \; ds.   
\ee
By the above, obviously $I (\lambda) = \frac{1}{2}\mathcal{I}_{\ell} (\lambda)$.
So our objective in this subsection is to compute $\mathcal{I}_m (\lambda)$.

Let us put
$$
f_m (s) = ( \cosh t - \cosh s )^m.
$$
Then we see easily that $f_m$ satisfies
\be\label{recurrence-relation-f-m}
f_m^{ '' } (s) = m^2 f_m (s) 
                    - m(2m-1) \cosh t   f_{m-1} (s)
                        +m(m-1) \sinh^2 t f_{m-2} (s),
\quad (m \geq 2).
\ee
By integration by parts and applying \eqref{recurrence-relation-f-m}, we get
\begin{equation*}
\begin{aligned}
\mathcal{I}_m (\lambda) 
  & \equiv  \int_{-t}^t \; \cos (\lambda s) \, f_m (s)  \; ds \\
  & = - \frac{1}{\lambda^2} \int_{-t}^t \; \cos (\lambda s) \, f_m^{ '' } (s)  \; ds \\
  & = - \frac{1}{\lambda^2} \int_{-t}^t \; \cos (\lambda s) \, \times \\
  &  \qquad \qquad
       \{ 
            m^2 f_m (s) 
               - m(2m-1) \cosh t   f_{m-1} (s)
                        +m(m-1) \sinh^2 t f_{m-2} (s)
          \} \, ds \\
  & = - \frac{m^2}{\lambda^2} \mathcal{I}_m (\lambda)
        - \frac{m(2m-1)}{\lambda^2} \cosh t \, \mathcal{I}_{m-1} (\lambda)
         + \frac{m(m-1)}{\lambda^2} \sinh^2 t \, \mathcal{I}_{m-2} (\lambda).
\end{aligned}
\end{equation*}
Therefore, we obtain the recurrence formula
\be\label{recurrence-relation-I-m}
\mathcal{I}_m (\lambda) 
  = \frac{1}{ \lambda^2 + m^2 } 
       \left\{
        -m(2m-1) \cosh t \, \mathcal{I}_{m-1} (\lambda) 
          + m(m-1) \sinh^2 t \, \mathcal{I}_{m-2} (\lambda)
             \right\}.
\ee
On the other hand, by direct computation, we have
\be\label{I-0-and-I-1}
\begin{aligned}
\mathcal{I}_0 (\lambda) 
  &=  \frac{ 2 \sin (\lambda t) }{\lambda}, \\
\mathcal{I}_1 (\lambda) 
  &= - \frac{ 2 \sinh t }{ \lambda^2 +1 } \cos (\lambda t)
     + \frac{ 2 \cosh t }{  (\lambda^2 +1 ) } \,
             \frac{ \sin (\lambda t) }{ \lambda }.
\end{aligned}
\ee
Combining \eqref{recurrence-relation-I-m} and \eqref{I-0-and-I-1},
we have

\begin{theorem}\label{expression-I-m-odd-case}
There exist rational functions $P_m (\lambda)$ and $Q_m (\lambda )$
of $\lambda$ such that
$$
\mathcal{I}_m (\lambda) 
= P_m (\lambda) \frac{\sin (\lambda t) }{ \lambda } 
   + Q_m (\lambda ) \cos (\lambda t).
$$
Moreover, the above $P_m (\lambda)$ and $Q_m (\lambda )$
are real valued and smooth on $\rr$.
\end{theorem}

We are now in a position to prove 
that
$\mathcal{I}_m (\lambda)$ is slowly decreasing.

By Theorem \ref{expression-I-m-odd-case},
$S_m (\lambda):= (P_m (\lambda)/ \lambda )^2 \, + Q_m (\lambda )^2$ 
is a rational function
which is smooth and non-negative on $(0, \infty)$.
$P_m$ and $Q_m$ may have a finite number of common zeros.
So if we take sufficiently large $\xi_0 (>0)$, then
$S_m (\lambda) >0$ for $\lambda > \xi_0$.
As a result, we have 
$$
S_m (\lambda)  \geq  \frac{A}{ (1+ | \lambda | )^k }, 
\qquad \text{for} \; \lambda \geq \xi_0,
$$
for some constants $A>0$ and $k$.

For $\xi \in \rr$, we define two subsets $U_{\xi}$ and $V_{\xi}$ of $\cc$ as follows.
\begin{equation*}
\begin{aligned}
U_{\xi} &:= \{ \zeta \in \cc ; || \zeta - \xi || < A \log (2 + | \xi | ) \}, \\
V_{\xi} &:= U_{\xi} \cap \rr.
\end{aligned}
\end{equation*}
Obviously, $U_{\xi} \supset V_{\xi}$.
If necessary, we take the above $\xi_0$ such that
\begin{equation*}
A \log (2 + | 2 \xi_0 | ) \geq \frac{ \pi }{ t }, \quad
2 \xi_0 - A \log (2 + | 2 \xi_0 | ) \geq \xi_0.
\end{equation*}
Then
\begin{equation*}
\begin{aligned}
\sup_{\zeta \in U_{\xi} } \, | \mathcal{I}_m (\zeta)  |
 & \geq \sup_{\lambda \in V_{\xi} } \, | \mathcal{I}_m (\lambda)  | \\
 & = \sup_{\lambda \in V_{\xi} } \,
    \left| \, 
     P_m (\lambda) \frac{\sin (\lambda t) }{ \lambda } 
   + Q_m (\lambda ) \cos (\lambda t)
    \right|     \\
 & = \sup_{\lambda \in V_{\xi} } \, S_m (\lambda) \\
 & \geq S_m (\xi)    \\
 & \geq \frac{A}{ (1+ | \xi | )^k }  ,
     \qquad \text{for} \; \xi \geq 2 \xi_0.
\end{aligned}
\end{equation*}
It follows easily from the above inequalities that
$\mathcal{I}_m (\zeta)$ 
is slowly decreasing.

%
%
%
%

\subsection{The case when $n$ is even}
Next, we consider the case when $n$ is an even number.
Our first objective in this subsection is to give an asymptotic expansion of $I (\lambda)$
as $\lambda \to +\infty$.

\medskip

We put $n-2 =2 \ell$ ($\ell \in \zz, \; \ell \geq 0$).
In this case, $I (\lambda)$ is written in terms of Bessel functions of the first kind.
More precisely, we have the following.


\begin{theorem}
\label{theorem-of-Bessel-function-series-expansion}
\begin{enumerate}
\item[(i)]
$I(\lambda)$ has the following {\it Bessel function series} expansion.
\begin{align}
\label{Bessel-function-series-expansion-0}
I(\lambda) &= 
          \sum_{m=0}^{N-1} \, d_m \, 
                         \left( \frac{t}{\lambda} \right)^{ \ell +m } \, J_{\ell +m} (\lambda t) 
                                \;  + r_N (\lambda),  \quad \text{where} \\
 \label{constant-d-m}
 d_m &=   \frac{\pi}{2}  ( 2 \ell + 2m -1 )!!  \times  
                       \sum_{j, k \geq 0, \,  \, j+k =m} \,  
                         a_0^{ \ell +k  - \frac{1}{2}  }  \, b_j^{ ( \ell +k ) } \, c_k.
\end{align}
The constants $a_0$, $b_j^{(\ell+k)}$, and $c_k$
are given respectively 
by the expressions \eqref{expression-of-a-0}, \eqref{definition-of-b-j-m}, 
and \eqref{definition-of-c-k} in Appendix A.
In addition, 
$L!!$ is defined as follows:
$(-1)!! =1$, $0!! =1$, and for any positive integer $L$, $L!!$ is 
the product of all the integers from $1$ up to $L$ that have the same parity as $L$.

\item[(ii)] The $N$-th remainder term $r_N (\lambda)$ satisfies
$$
| r_N (\lambda) | \leq \frac{ C_N }{ (|\lambda|+1)^N  }, \qquad
\text{for} \; \lambda \in \rr.
$$
\end{enumerate}
\end{theorem}

We will prove the above theorem in Appendix A.

Formula \eqref{Bessel-function-series-expansion-0} gives an explicit asymptotic expansion
of $I(\lambda)$ as $\lambda \to \infty$, because the asymptotic expansion of 
each Bessel function is well known.
In fact, we have

\begin{theorem}(See, for example, \cite{Gradshteyn-Ryzhik2015}, \S8.451, Formula 1.)
\label{asymptotic-expansion-of-Bessel-function}
If $z \to \infty$ under the condition that $ | \arg z | < \pi$, we have
\begin{equation*}
\begin{aligned}
J_m (z) &= \sqrt{ \frac{2}{\pi z}  } \cos \left( z - \frac{\pi}{2} m - \frac{\pi}{4} \right)        \\
           & \quad \times 
               \left\{ \; \sum_{ k=0 }^{ N-1 }
             \frac{ (-1)^k \Gamma (m+2k + \frac{1}{2} )  }{  (2k)!  \Gamma (m-2k + \frac{1}{2} )  }
                  \frac{1} { (2z)^{2k} } \; \; + R_N^{(1)}          \right\}                             \\
           & - \sqrt{ \frac{2}{\pi z}  } \sin \left( z - \frac{\pi}{2} m - \frac{\pi}{4}    \right)   \\
           & \quad \times 
               \left\{ \; \sum_{ k=0 }^{ N-1 }
             \frac{ (-1)^k \Gamma (m+2k + \frac{3}{2} )  }{  (2k+1)!  \Gamma (m-2k - \frac{1}{2} )  }
                  \frac{1} { (2z)^{2k+1} } \; \; + R_N^{(2)}        \right\},   
\end{aligned}
\end{equation*}
where the $N$-th remainder terms $R_N^{(1)}$ and $R_N^{(2)}$ satisfy
\begin{equation*}
\begin{aligned}
| R_N^{(1)} |
    & <
      \left| 
          \frac{ (-1)^N \Gamma (m+2N + \frac{1}{2} )  }{  (2N)!  \Gamma (m-2N + \frac{1}{2} ) (2z)^{2N} } 
             \right|,  \quad (N > \frac{m}{2} - \frac{1}{4} )  \\
| R_N^{(2)} |
    & <
      \left|  
          \frac{ (-1)^N \Gamma (m+2N + \frac{3}{2} )  }{  (2N+1)!  \Gamma (m-2N - \frac{1}{2} ) (2z)^{2N+1} } 
             \right|, \quad (N > \frac{m}{2} - \frac{3}{4} ).
\end{aligned}
\end{equation*}
\end{theorem}
 
Theorem \ref{theorem-of-Bessel-function-series-expansion} 
and Theorem \ref{asymptotic-expansion-of-Bessel-function}
yield the following.

\begin{corollary}
$I (\lambda)$ is written as
\begin{equation*}
\begin{aligned}
I (\lambda) &=  d_0 \left( \frac{t}{\lambda} \right)^{ \ell } 
                       \times \sqrt{ \frac{2}{\pi \lambda t}  }  
                         \cos \left( \, \lambda t - \frac{\pi}{2} \ell - \frac{\pi}{4} \right) 
                            +  \widetilde{r_1} (\lambda), \\
\text{where } & \widetilde{r_1} (\lambda) \text{ satisfies}  \\
| \widetilde{r_1} (\lambda) | & \leq \frac{C}{ ( | \lambda | +1 )^{ \ell +\frac{3}{2}  } },
\qquad \lambda \in \rr.
\end{aligned}
\end{equation*}
\end{corollary}

Now we will prove that
$I (\lambda)$ 
is slowly decreasing.

For $\xi \in \rr$, we define three subsets $U_{\xi}$, $V_{\xi}$, and $W_{\xi}$ of $\cc$ as follows.
\begin{equation*}
\begin{aligned}
U_{\xi} &:= \{ \zeta \in \cc ; || \zeta - \xi || < A \log (2 + | \xi | ) \}, \\
V_{\xi} &:= U_{\xi} \cap \rr, \\
W_{\xi} &:= \{ \, \lambda \in \rr ; \xi - 2 \pi < \lambda t < \xi  \,  \}, 
\end{aligned}
\end{equation*}
Obviously, $U_{\xi} \supset V_{\xi}$.
We take 
\begin{equation*}
A= \sqrt{ \frac{1}{2 \pi t }  }  d_0  t^{\ell}.
\end{equation*}
Then we see easily that
\begin{equation*}
\sup_{\lambda \in W_{\xi} } \,
   \left| 
     d_0 \left( \frac{t}{\lambda} \right)^{ \ell } 
                       \times \sqrt{ \frac{2}{\pi \lambda t}  }  
                         \cos \left( \, \lambda t - \frac{\pi}{2} \ell - \frac{\pi}{4} \right) 
       \right| \geq 2A | \xi |^{- \ell - \frac{1}{2}  }
\quad \text{for} \; \, \xi > 2 \pi.
\end{equation*}
Next, we take sufficiently large positive constant $\delta_0 (> 2 \pi)$. 
Then for $\xi > \delta_0$, we have
$V_{\xi} \supset W_{\xi}$.
(Take $\delta_0$ such that $A \log (2 + | \delta_0 | ) \geq 2 \pi t$.)
As a result, we have
\begin{equation*}
U_{\xi} \supset V_{\xi} \supset W_{\xi}, \quad \text{for} \; \, \xi > \delta_0.
\end{equation*}
Therefore, we have
\begin{equation*}
\begin{aligned}
\sup_{\zeta \in U_{\xi} } \, | I (\zeta) |
 & \geq \sup_{\lambda \in V_{\xi} } \, | I (\lambda) | \\
 & \geq \sup_{\lambda \in W_{\xi} } \, | I (\lambda) | \\
 & \geq 
   \sup_{\lambda \in W_{\xi} } \,
   \left| 
     d_0 \left( \frac{t}{\lambda} \right)^{ \ell } 
                       \times \sqrt{ \frac{2}{\pi \lambda t}  }  
                         \cos \left( \, \lambda t - \frac{\pi}{2} \ell - \frac{\pi}{4} \right) 
       \right| 
         - \sup_{\lambda \in W_{\xi} } \, | \widetilde{r_1} (\lambda) |  \\
 & \geq 
    2A | \xi |^{- \ell - \frac{1}{2}  } -  \frac{C}{ ( | \xi | +1 )^{ \ell + \frac{3}{2}  } } \\
 & \geq 
  A ( | \xi | +1)^{ -\ell - \frac{1}{2} } 
    + ( | \xi | +1)^{ -\ell - \frac{1}{2} } 
        \left( A- C ( | \xi | +1 )^{ -1  } \right),
          \quad \text{for} \; \, \xi > \delta_0.
\end{aligned}
\end{equation*}
Again, we take another positive constant $\delta_1 (> \delta_0)$ such that
$$
A- C ( | \xi | +1 )^{ -1  } > 0, \quad \text{for} \; \, \xi > \delta_1.
$$
Then we have
$$
\sup_{\zeta \in U_{\xi} } \, | I (\zeta) | \geq A ( | \xi | +1)^{ -\ell - \frac{1}{2} } 
   \quad \text{for} \; \, \xi > \delta_1.
$$
A similar argument holds for $\xi < - \delta_1$.
Namely, we have
\be\label{weak-Ehrenpreis-estimate}
\sup_{\zeta \in U_{\xi} } \, | I (\zeta) | \geq A ( | \xi | +1)^{ -\ell - \frac{1}{2} } 
   \quad \text{for} \; \,  | \xi | > \delta_1.
\ee
It follows from \eqref{weak-Ehrenpreis-estimate} that
$I (\lambda)$ 
is slowly decreasing.

%
%
%
%
%

\section{Appendix A: Bessel series expansion}

In this appendix, we will prove Theorem \ref{theorem-of-Bessel-function-series-expansion}.

We put
\begin{align*}
F(z) &:= 
        {}_2 F_1 
      \left( 
 1- \frac{q}{2}, \; \frac{q}{2} \; ;  \; \frac{n-1}{2}  \;  ;
 \frac{ z }{ 2\cosh t }
     \right), \\
Z(s) &:= \cosh t - \cosh s.
\end{align*}
Let us write the power series expansion of $F(z)$ as
$$
F(z) = \sum_{k=0}^{\infty} \, c_k z^k 
      = \sum_{k=0}^{N-1} \, c_k z^k \; + \, F_N (z),
$$
where $F_N$ is the $N$-th remainder term and the $k$-th coefficient
$c_k$ is given by
\be\label{definition-of-c-k}
c_k = 
\frac{ (1- \frac{q}{2} )_k ( \frac{q}{2} )_k }{ (\frac{n-1}{2} )_k \, k! \, (2 \cosh t)^k }.
\ee
Then $I (\lambda)$ is written as
\be\label{expansion-of-Fourier-integral}
\begin{aligned}
I (\lambda) 
 &=
  \sum_{k=0}^{N-1} c_k \, \int_0^t \; \cos (\lambda s) \;
   \{ Z(s) \}^{ \ell +k - \frac{1}{2} } \, ds \\
&  \; \; +  \int_0^t \; \cos (\lambda s) \; 
       \{ Z(s) \}^{ \ell  - \frac{1}{2} } \, F_N ( Z(s) ) \, ds.
\end{aligned}
\ee
Next, we put
\be\label{k-th-Fourier-integral}
I_k (\lambda) :=
 \int_0^t \; \cos (\lambda s) \;
   \{ Z(s) \}^{ \ell +k - \frac{1}{2} } \, ds
\ee
From now on, we will expand the right hand side of \eqref{k-th-Fourier-integral}
as a series of Bessel functions.
Let us define a function $f$ by
$$
f(z) = \sum_{k=0}^{\infty} \, \frac{1}{(2k)!} z^k
$$
Then we see easily that $\cosh z = f(z^2)$.
Moreover, we have
\begin{align*}
{}& \frac{ f(t^2) -f(t^2-z)  }{ z } = \sum_{k=0}^{\infty} \,  a_k z^k, 
  \quad \text{where} \\
{}& a_k = \frac{ (-1)^k f^{(k+1)} (t^2)  }{ (k+1)!  }, \; \; \; (k=0,1,2, \cdots).
\end{align*}
Therefore, we can put
\begin{align}
\left\{
 \frac{ f(t^2) -f(t^2-z)  }{ z } 
  \right\}^{ m - \frac{1}{2}  } 
&= a_0^{ m - \frac{1}{2}  }
    \left( \, 1 + \frac{a_1}{a_0} z + \frac{a_2}{a_0} z^2 + \cdots \right)^{ m - \frac{1}{2}  }  \\
\label{definition-of-b-j-m}
&= a_0^{ m - \frac{1}{2}  } \,
    ( 1 + b_1^{(m)} z + b_2^{(m)} z^2 + \cdots ) \\
&= a_0^{ m - \frac{1}{2}  } \,
    ( 1 + b_1^{(m)} z + b_2^{(m)} z^2 + \cdots + b_{N-1}^{(m)} z^{N-1}) 
      + g_{(m, N)} (z) z^N,
  \label{expansion-of-fractional-power-of-function}
\end{align}
where $g_{(m, N)} (z)$ is holomorphic near the line segment $[0, t^2] \subset \cc$
and where the coefficients $b_j^{(m)} \; (j=1,2,3 \cdots)$ are written as a polynomial
of $\displaystyle{\frac{a_1}{a_0},  \frac{a_2}{a_0}, \cdots, \frac{a_j}{a_0}  }$.
For example, the first three coefficients 
$b_1^{(m)}, b_2^{(m)}, b_3^{(m)}$ in the above expansion
are given by
\begin{align*}
b_1^{(m)} &= ( m - \frac{1}{2} ) \frac{a_1}{a_0}, \quad
  b_2^{(m)} =  ( m - \frac{1}{2} ) \frac{a_2}{a_0}
                    + \begin{pmatrix} m - \frac{1}{2} \\ 2 \end{pmatrix}
                       \left( \frac{a_1}{a_0} \right)^2, \\
b_3^{(m)} &=  ( m - \frac{1}{2} ) \frac{a_3}{a_0}
                   + 2 \begin{pmatrix} m - \frac{1}{2} \\ 2 \end{pmatrix}
                      \frac{a_1 a_2}{a_0^2}
                        + \begin{pmatrix} m - \frac{1}{2} \\ 3 \end{pmatrix}
                           \left( \frac{a_1}{a_0} \right)^3.
\end{align*}
Here we define $b_0^{(m)} :=1$.
We also note that for each fixed $t >0$ 
\be\label{expression-of-a-0}
a_0 = \frac{  \sinh \sqrt{t} }{ 2 \sqrt{t} } >0.
\ee
By substituting $z= t^2-s^2$ in \eqref{expansion-of-fractional-power-of-function}, 
we have
\begin{equation*}
\begin{aligned} 
 \left( \frac{ \cosh t - \cosh s  }{ t^2 - s^2  } \right)^{ m - \frac{1}{2}  } 
 &= a_0^{ m - \frac{1}{2}  } \, 
     \sum_{j=0}^{N-1} \,
        b_j^{(m)} (t^2 -s^2 )^j   \\
        &\qquad + g_{(m, N)} (t^2 -s^2 ) (t^2 -s^2 )^N.
\end{aligned}
\end{equation*}
Thus we have
\begin{equation*}
\begin{aligned} 
{}& \{ Z(s) \}^{ \frac{ n-3 }{2} +k } \equiv \{ Z(s) \}^{ \ell+k - \frac{1}{2}  } \\
{}& \equiv ( \cosh t - \cosh s )^{ \ell+k - \frac{1}{2}  } \\
{}& \; = (t^2 -s^2 )^{ \ell+k - \frac{1}{2}  } \;
           \left( \frac{ \cosh t - \cosh s  }{ t^2 - s^2  } \right)^{ \ell +k - \frac{1}{2}  } \\
{}& \; =  a_0^{ \ell +k  - \frac{1}{2}  } \, 
              \sum_{j=0}^{N-1} \,  b_j^{( \ell +k)}      
                (t^2 -s^2 )^{ \ell+k+j - \frac{1}{2}  } \\
{}& \; \;  + (t^2 -s^2 )^{ \ell+k+N - \frac{1}{2}  } \, g_{( \ell+k , N)} (t^2 -s^2 ) 
\end{aligned}
\end{equation*}
Therefore, $I_k (\lambda)$ is rewritten as
\be\label{expansion-of-k-th-Fourier-integral}
\begin{aligned} 
I_k (\lambda) 
   & = a_0^{ \ell +k  - \frac{1}{2}  } \, 
            \sum_{j=0}^{N-1} \,  b_j^{( \ell +k)}
             \int_0^t \, \cos ( \lambda s ) (t^2 -s^2 )^{ \ell+k+j - \frac{1}{2}  } \; ds \\
    & \; \; + \int_0^t \, \cos ( \lambda s ) 
                   (t^2 -s^2 )^{ \ell+k+N - \frac{1}{2}  } \, g_{( \ell+k , N)} (t^2 -s^2 ) \; ds.
\end{aligned}
\ee
Here in the R. H. S. of \eqref{expansion-of-k-th-Fourier-integral}, we have
\be\label{integral-formula-for-Bessel-function}
\begin{aligned} 
{} & \int_0^t \, \cos ( \lambda s ) (t^2 -s^2 )^{ \ell+k+j - \frac{1}{2}  } \; ds \\
{} & = t^{ 2( \ell+k+j )  } \,
          \int_0^{ \frac{\pi}{2} } \; 
             \cos ( \lambda t \sin \theta ) \, \cos^{ 2( \ell+k+j ) } \theta \, d \theta \\
{} & = \frac{\pi}{2} \times  ( 2 \ell + 2k + 2j -1 )!!  \times 
            \left( \frac{ t }{ \lambda } \right)^{ \ell+k+j  } \times 
              J_{ \ell+k+j }  (\lambda t ).
\end{aligned}
\ee
In the above, $J_m (z)$ denotes the Bessel function of the first kind.
We also note that in the computation of \eqref{integral-formula-for-Bessel-function},
we used the following formula for Bessel functions.
\begin{equation*}
\int_0^{ \frac{\pi}{2} } \; 
  \cos ( z \sin \theta ) \, \cos^{ 2m } \theta \, d \theta
   =  \frac{\pi}{2} \times  \frac{ (2m-1)!! }{ z^m } \times J_m (z).
\end{equation*}
For details, see, for example, \cite{Gradshteyn-Ryzhik2015}, \S3.715, Formula 10.

Combining \eqref{expansion-of-Fourier-integral}, 
\eqref{k-th-Fourier-integral}, 
\eqref{expansion-of-k-th-Fourier-integral},
and \eqref{integral-formula-for-Bessel-function}, 
we have
\be\label{series-of-Bessel-functions-1}
\begin{aligned} 
{}& I (\lambda ) \\
  &= \sum_{k=0}^{N-1} \, \sum_{j=0}^{N-1} \, 
     a_0^{ \ell +k  - \frac{1}{2}  }  \, b_j^{ ( \ell +k ) } \, c_k 
      \times \frac{\pi}{2} \times  ( 2 \ell + 2k + 2j -1 )!!  \times 
            \left( \frac{ t }{ \lambda } \right)^{ \ell+k+j  } \times 
              J_{ \ell+k+j }  (\lambda t ) \\
  & \; \; + \sum_{k=0}^{N-1} \, I_k^N (\lambda) + R_N (\lambda),
\end{aligned}
\ee
where
\begin{align}
\label{remainder-term-1}
I_k^N (\lambda) 
    &= 
       \int_0^t \, \cos ( \lambda s ) 
            (t^2 -s^2 )^{ \ell+k+N - \frac{1}{2}  } \, g_{( \ell+k , N)} (t^2 -s^2 ) \; ds, \\
\label{remainder-term-2}
R_N (\lambda) 
    &= 
       \int_0^t \; \cos (\lambda s) \; 
           \{ Z(s) \}^{ \ell  - \frac{1}{2} } \, F_N ( Z(s) ) \, ds.
\end{align}
Let us rewrite \eqref{series-of-Bessel-functions-1} as follows.
\be\label{series-of-Bessel-functions-2}
\begin{aligned}
I (\lambda) &= \sum_{m=0}^{N-1} \, d_m \, 
                         \left( \frac{t}{\lambda} \right)^{ \ell +m } \, J_{\ell +m} (\lambda t) 
                       + \widetilde{R_N} (\lambda), \\
 \text{where} &  \\
    d_m &=   \frac{\pi}{2}  ( 2 \ell + 2m -1 )!!  \times  
                       \sum_{0 \leq j, k  \, j+k =m} \,  
                         a_0^{ \ell +k  - \frac{1}{2}  }  \, b_j^{ ( \ell +k ) } \, c_k, \\                
\widetilde{R_N} (\lambda) &= 
                \frac{\pi}{2}  \, \sum_{0 \leq j, k \leq N-1,  \, N \leq j+k} \, 
                   a_0^{ \ell +k  - \frac{1}{2}  }  \, b_j^{ ( \ell +k ) } \, c_k 
                     ( 2 \ell + 2k + 2j -1 )!!  \\
     & \qquad \qquad \qquad \qquad \times
             \left( \frac{ t }{ \lambda } \right)^{ \ell+k+j  } \,  J_{ \ell+k+j }  (\lambda t ) \\
     & \; \; + \sum_{k=0}^{N-1} \, I_k^N (\lambda) + R_N (\lambda)
\end{aligned}
\ee
Our next objective is to estimate the remainder term $\widetilde{R_N} (\lambda)$,
We need the following lemma.

\begin{lemma}\label{estimate-for-remainder-term}
\begin{enumerate}
\item[(i)]
Let $\varphi (s)$ be an even function of class $C^{2m}$ defined on $[-t, t]$.
We assume that 
\be
\varphi (\pm t) = \varphi' (\pm t) =\varphi'' (\pm t) = \cdots = \varphi^{(2m-1)} (\pm t) =0.
\ee
Then we have
$$
\int_0^t \, \cos (\lambda s) \varphi (s) \, ds 
= \frac{ (-1)^m }{ \lambda^{2m} } \int_0^t \, \cos (\lambda s) \varphi^{(2m)} (s) \, ds.
$$
\item[(ii)] Moreover, we have
$$
\Bigl| \, \int_0^t \, \cos (\lambda s) \varphi (s) \, ds  \, \Bigr| 
 \; \leq \; 
 \frac{ C  }{ | \lambda  |^{2m}  }, 
     \qquad \text{for} \; \lambda \in \rr \setminus \{ 0 \},
$$
\end{enumerate}
where the above constant $C$ does not depent on $\lambda$.
\end{lemma}

\begin{proof}
By repeating integral by parts $2m$-times, we get (i).
By taking $C= \displaystyle{  \int_0^t \, | \varphi^{(2m)} (s) | \, ds }$, we get the estimate (ii).
\end{proof}

We apply Lemma \ref{estimate-for-remainder-term} to the remaider terms
$I_k^N (\lambda)$ and $R_N (\lambda)$
by taking
\begin{align*}
\varphi (s) &=    
       (t^2 -s^2 )^{ \ell+k+N - \frac{1}{2}  } \, g_{( \ell+k , N)} (t^2 -s^2 ), \\
\text{and} \; \;
\varphi (s) &= 
       \{ Z(s) \}^{ \ell  - \frac{1}{2} } \, F_N ( Z(s) ), 
\end{align*}
respectively.
As a result, we have
\begin{align}
\label{estimate-for-I-k}
| I_k^N (\lambda) | & \leq    
       \frac{ C_{ (k, N)} }{ | \lambda  |^{2[ \frac{\ell+k+N-1}{2} ] }  }, 
         \qquad \text{for} \; \lambda \in \rr \setminus \{ 0 \}, \\
\label{estimate-for-R-N}
\text{and} \; \;
| R_N (\lambda) | & \leq 
     \frac{ C_{(N)}  }{ | \lambda  |^{2[  \frac{\ell+N-1}{2} ] }  }, 
        \qquad \text{for} \; \lambda \in \rr \setminus \{ 0 \},
\end{align}
where $[x]$ denotes the largest integer less than or equal to $x$.
In the above estimates, the constants $C_{ (k, N)}$ and $C_{(N)}$ do not depend
on $\lambda$.
Next, we go into the estimate of Bessel functions.
By the definition of the Bessel function
$$
J_m (z) = \frac{1}{\pi} \int_0^{\pi} \, \cos (m \theta - z \sin \theta ) \, d \theta,
$$
we have
\be\label{estimate-for-Bessel-function-1}
J_m (\lambda t) \leq 1, \qquad \text{for} \; \lambda \in \rr.
\ee
By 
\eqref{estimate-for-I-k},
\eqref{estimate-for-R-N},
and
\eqref{estimate-for-Bessel-function-1}, we have
\be\label{estimate-for-widetilde-R-N}
|  \widetilde{R_N} (\lambda)  | 
    \leq 
       \frac{\widetilde{C_N} }{ (|\lambda|+1)^{ 2[ \frac{N-1}{2} ]  }  },
\qquad \text{for} \; \lambda \in \rr,
\ee
if we take some suitable constant $\widetilde{C_N}$.

Here we replace $N$ by $2N+1$ in \eqref{series-of-Bessel-functions-2}. 
Then we have
\begin{align}
  \label{series-of-Bessel-functions-3}
I (\lambda) &= \sum_{m=0}^{N-1} \, d_m \, 
                         \left( \frac{t}{\lambda} \right)^{ \ell +m } \, J_{\ell +m} (\lambda t) 
                           + r_N (\lambda),  \; \; \text{where} \\   
  \label{remainder-term-r-N}
      r_N (\lambda) 
                &= \sum_{m=N}^{2N} \, d_m \, 
                        \left( \frac{t}{\lambda} \right)^{ \ell +m } \, J_{\ell +m} (\lambda t)
                           + \widetilde{R_{2N+1} } (\lambda).
\end{align}
We apply \eqref{estimate-for-Bessel-function-1} and
\eqref{estimate-for-widetilde-R-N} to R.H.S of
\eqref{remainder-term-r-N}. Then we see easily that 
there exists a positive constant
$C_N$ independent of $\lambda$ such that
$$
| r_N (\lambda) | \leq \frac{ C_N }{ (|\lambda|+1)^N  }, \qquad
\text{for} \; \lambda \in \rr.
$$

This finishes the proof of
Theorem \ref{theorem-of-Bessel-function-series-expansion}.
%
%

\section{Appendix B: A Smoothness Result}

In this section we will prove a smoothness result for quotients that is needed to complete the proof of Theorem \ref{T:surjectivity}.

We first introduce some notation which we will use below.  For any $z\in\cc$ and $r>0$, let $D_r(z)=\{\zeta\in\cc\,:\,|\zeta-z|\leq r\}$, let $\overline D_r(z)$ be its closure, and let $C_r(z)$ be the circle $\{\zeta\in\cc\,:\,|\zeta-z|=r\}$.  If $z=0$, we will just use the notation $D_r$ and $C_r$ in place of $D_r(0)$ and $C_r(0)$.

\begin{proposition}\label{T:smoothness}
Let $U\subset \rr^m$ and  $V\subset\cn$ be open sets, and let $F(x,z)$ be $C^\infty$ on $U\times V$ and holomorphic in $z$ for each fixed $x$.  Suppose that $g(z)$ is a nonzero holomorphic function on $V$ such that $F(x,z)/g(z)$ is holomorphic for each $x$.  Then $F(x,z)/g(z)$ is $C^\infty$ on $U\times V$. 
\end{proposition}

\begin{proof}
  It suffices to prove that $F/g$ is smooth on a neighborhood of each point $(x_0,z_0)\in U\times V$ for which $g(z_0)=0$.   For simplicity, we can assume that $x_0=0$ and $z_0=0$ and then by shrinking $U$ and $V$, we can assume that $U$ and $V$ are balls centered at the origins in $\rr^m$ and $\cc^n$, respectively.   We can also assume that $F$ is not identically $0$.

Let $\mathcal V\subset V$ be the zero locus $\mathcal V=\{z\in V\,:\,g(z)=0\}$.  Since $g$ is not identically $0$, it is not identically $0$ on some complex line $\ell$ through the origin in $V$; applying a linear operator on $\cn$, we can assume that $\ell=\{z\in V\,:\,z_2=\cdots=z_n=0\}$.  (If $n=1$, we have $\ell=V$.)   

Since we're trying to prove the smoothness of $F/g$ on a neighborhood of $(0,0)$, we can conveniently  shrink $V$ even further and assume that $V$ is an open polydisk $D_{r_1}\times\cdots\times D_{r_n}$ in $\cn$, so $\ell=D_{r_1}\times\{0\}\times\cdots\times\{0\}$.  

Now $\ell\cap \mathcal V$ consists of isolated points, so in particular there is closed disk $\overline D_{s_1}$ (with positive radius $s_1$) inside $D_{r_1}$ such that
$\mathcal V\cap (D_{s_1}\times\{0\}\times\ldots\times\{0\})=(0,\ldots,0)$.

 Since $C_{s_1}\times\{0\}\times\cdots\times\{0\}$ is a compact set disjoint from the (relatively) closed set $\mathcal V$,
there is a tube $C_{s_1}\times\overline D_{s_2}\times\cdots\times\overline D_{s_n}$ in $V$ disjoint from $\mathcal V$.
Then in particular, $C_{s_1}\times C_{s_2}\times\cdots\times C_{s_n}$ is disjoint from $\mathcal V$.  Let $V'=D_{s_1}\times\cdots\times D_{s_n}$.

Now for any $(x,z)\in U\times V'$, we have
\begin{multline*}
\frac{F(x,z)}{g(z)}\\
=\frac 1{(2\pi i)^n}\,\int_{C_{s_1}}\cdots\int_{C_{s_n}}\frac{F(x,\zeta_1,\ldots,\zeta_n)/g(\zeta_1,\ldots,\zeta_n)}
{(\zeta_1-z_1)\cdots (\zeta_n-z_n)}\,d\zeta_n\cdots d\zeta_1.
\end{multline*}
Since $|g(\zeta_1,\ldots,\zeta_n)|$ is bounded below by a fixed positive constant on the compact set $C_{s_1}\times\cdots\times C_{s_n}$, we can apply the usual arguments for differentiating inside the integral to conclude that if $\alpha\in (\zz^+)^m$ and $\beta=(\beta_1,\ldots,\beta_n)\in (\zz^+)^n$, then
\begin{multline*}
D_x^\alpha\,(\partial_{z_1}^{\beta_1})\cdots(\partial_{z_n}^{\beta_n})\,\left(\frac{F(x,z)}{g(z)}\right)\\
=\frac{\beta_1!\cdots\beta_n!}{(2\pi i)^n}\,\int_{C_{s_1}}\cdots\int_{C_{s_n}}\frac{D_x^\alpha F(x,\zeta_1,\ldots,\zeta_n)/g(\zeta_1,\ldots,\zeta_n)}
{(\zeta_1-z_1)^{\beta_1+1}\cdots (\zeta_n-z_n)^{\beta_n+1}}\,d\zeta_n\cdots d\zeta_1
\end{multline*}
for all $(x,z)\in U\times V'$.  This of course proves the lemma.
\end{proof}

The following example shows that analyticity is needed in the second argument. Let $m=n=1$ and define the function $g(y)$ on $\rr$ by
$$
g(y)=\begin{cases}
e^{-1/y^2}&\text{ if }y\neq 0\\
0&\text{ if }y=0.
\end{cases}
$$
Then the function on $\rr^2$ given by
$$
F(x,y)=\begin{cases}
xg(y)/(x^2+y^2)&\text{ if }(x,y)\neq (0,0)\\
0&\text{ if }(x,y)=(0,0)
\end{cases}
$$
is smooth, but $F(x,y)/g(y)$ does not extend to a continuous function on $\rr^2$.

\bibliographystyle{alpha}
\bibliography{meanvalue}
%
%
%
%
%
%
%
%
%
%
\end{document}